\documentclass{siamltex}
\usepackage{amsmath}
\usepackage{amsfonts}
\usepackage{amssymb}
\usepackage{graphicx}
\usepackage{color}
\usepackage{hyperref}

\hypersetup{
	colorlinks=true,   
	urlcolor=blue,
	citecolor=blue,   
	linkcolor=blue}
\usepackage{bm}
\usepackage{verbatim}

\usepackage{bmpsize}
\usepackage{algorithm}
\usepackage[noend]{algpseudocode}
\allowdisplaybreaks

\usepackage[left=1.5in, top=1in, right=1in, bottom=1in]{geometry}

\newtheorem{Aalgorithm}[theorem]{Algorithm}

\newcommand{\be}{\begin{equation}}
	\newcommand{\ee}{\end{equation}}
\newcommand{\bea}{\begin{eqnarray}}
	\newcommand{\eea}{\end{eqnarray}}
\newcommand{\beas}{\begin{eqnarray*}}
	\newcommand{\eeas}{\end{eqnarray*}}

\newcommand{\bfu}{\ensuremath{\mathbf{u}}}

\newcommand{\vertiii}[1]{{\left\vert\kern-0.25ex\left\vert\kern-0.25ex\left\vert #1 
		\right\vert\kern-0.25ex\right\vert\kern-0.25ex\right\vert}}

\makeatletter
\def\BState{\State\hskip-\ALG@thistlm}
\makeatother

\graphicspath{ {./result/} }
\usepackage{booktabs}
\begin{document}
\title{A family of second order, linear, unconditionally stable methods for the Cahn-Hilliard-Navier-Stokes equations}
	
	\author{
    Daozhi Han$^{\text{a}}$, Nan Jiang$^{\text{b},\dag}$, Jonah H. Nissan$^{\text{c}}$, Sayantan Sarkar$^{\text{a},\ast}$\\[2ex]
    $^{\text{a}}$\textit{Department of Mathematics, State University of New York at Buffalo, Buffalo, NY 14260}\\
    $^{\text{b}}$\textit{Department of Mathematics, University of Florida, Gainesville, FL 32611}\\
    $^{\text{c}}$\textit{College of Liberal Arts and Sciences, University of Florida, Gainesville, FL 32611}
}

\maketitle

\renewcommand{\thefootnote}{\fnsymbol{footnote}}
\footnotetext[1]{Corresponding author. \\
\indent \textit{Email addresses:} \texttt{daozhiha@buffalo.edu} (Daozhi Han), \texttt{jiangn@ufl.edu} (Nan Jiang), \texttt{jonah.nissan@ufl.edu} (Jonah H. Nissan), \texttt{sayantansark@umass.edu} (Sayantan Sarkar)}
\footnotetext[2]{Nan Jiang and Jonah H. Nissan were partially supported by the US National Science Foundation grant DMS-2143331.}
\renewcommand{\thefootnote}{\arabic{footnote}} %
   
    \vspace{2ex}
    \noindent\rule{\textwidth}{0.4pt}
    \vspace{1ex}

    \noindent \textbf{Abstract} \\[1ex]

We present a family of second-order, linear, unconditionally stable implicit-explicit (IMEX) methods for the Cahn-Hilliard-Navier-Stokes (CHNS) equations modeling matched-density two-phase flows. The proposed semi-discrete scheme combines extrapolation of the nonlinear terms with an auxiliary-variable formulation of the nonlinear free-energy term and a temporal-curvature regularization controlled by a parameter $\epsilon$. We establish a discrete energy estimate showing unconditional long-time stability of the method for $\theta\in(1/2,1]$ and $\epsilon\geq0$. The resulting scheme requires only linear solves at each time step. Numerical experiments demonstrate approximately second-order temporal convergence and examine mass conservation, energy dissipation, numerical robustness, and several representative interfacial-flow problems, including spinodal decomposition, droplet shape relaxation, two-phase lid-driven cavity flow, and Rayleigh--Taylor instability.

    \vspace{2ex}
    \noindent \textit{Keywords:} phase field models,  Cahn-Hilliard-Navier-Stokes, finite element method, unconditional stability

    \vspace{1ex}
    \noindent\rule{\textwidth}{0.4pt}
    \vspace{4ex}

\section{Introduction}

The accurate mathematical modeling and computational simulation of complex multiphase fluid dynamics remain challenging problems in computational mechanics. Classical approaches to immiscible fluid interactions include sharp-interface and interface-capturing methods, such as the Volume of Fluid and Level-Set methods, which require dedicated procedures for interface advection, reconstruction, or tracking \cite{Anderson1998, Magaletti2013}. Alternatively, diffuse-interface (phase-field) models represent material interfaces through smooth transitions across a finite-thickness interfacial region \cite{Lowengrub1998, AbelsLengeler2014}. The Cahn-Hilliard-Navier-Stokes (CHNS) equations provide a canonical framework for such systems \cite{LiuShen2003, Boyer1999}.

By introducing a phase function $\phi(x,t)$ that continuously represents the mixture composition, the CHNS system naturally accommodates topological transitions, such as droplet coalescence and pinch-off, without requiring explicit interface reconstruction. The phase-field dynamics are governed by a Ginzburg-Landau-type free-energy functional \cite{Abels2009}:
\begin{equation}
    W(\phi,\nabla\phi)
    =
    \int_{\Omega}
    \lambda
    \left(
    \frac{1}{2}|\nabla\phi|^{2}
    +
    F(\phi)
    \right)
    dx,
    \label{eq:free_en}
\end{equation}
where
\[
F(\phi)
=
\frac{1}{4\eta^{2}}(\phi^{2}-1)^{2}
\]
represents the double-well bulk energy, while the gradient term penalizes spatial variations of the phase field and contributes to the interfacial energy \cite{GiorginiMiranville2019}. The governing equations couple the incompressible Navier-Stokes momentum balance with the fourth-order Cahn-Hilliard convection-diffusion equation. The mathematical analysis of this class of systems includes thermodynamically consistent formulations and well-posedness results \cite{AbelsGarckeGrun2012}, global and trajectory attractors \cite{GalGrasselli2010, GiorginiTemam2022, FrigeriGrasselli2012}, and convergence-to-equilibrium results based on \L{}ojasiewicz-Simon inequalities \cite{ZhaoWuHuang2009}. Complex boundary dynamics have also been modeled by incorporating generalized Navier boundary conditions (GNBC) and dynamic contact-angle effects \cite{QianWangSheng2003, QianWangSheng2006, YueFeng2011, GiorginiKnopf2023}. In the absence of external forcing, the continuous CHNS system satisfies the energy-dissipation law \cite{Abels2009, LiuShen2003}:
\begin{equation}
    \frac{d}{dt}
    \left\{
    W(\phi,\nabla\phi)
    +
    \int_{\Omega}\frac{1}{2}|u|^{2}\,dx
    \right\}
    =
    -\int_{\Omega}
    \left(
    M|\nabla\mu|^{2}
    +
    \nu|\nabla u|^{2}
    \right)
    dx.
    \label{eq:en_law}
\end{equation}
Preserving an appropriate counterpart of this dissipation structure at the discrete level is an important consideration in the design of numerical schemes for the CHNS system \cite{Feng2006, HanWang2017}.

The mathematical stiffness of the fourth-order Cahn-Hilliard operator, together with its nonlinear coupling to the fluid equations, can make straightforward explicit time-stepping severely restrictive. In particular, explicit treatments of the fourth-order diffusion operator can lead to time-step restrictions as severe as $\Delta t \sim \mathcal{O}(\Delta x^{4})$ \cite{GuillenGonzalez2013, DongShen2010}. Consequently, researchers have developed a wide range of spatial and temporal discretization strategies for phase-field flow models.

Spatially, primal conforming formulations of the fourth-order operator require globally $C^{1}$-continuous basis functions. While Isogeometric Analysis naturally provides the required smoothness \cite{Gomez2008, Hosseini2016}, classical finite element spaces with global $C^{1}$ continuity are comparatively difficult to construct, particularly in three dimensions. Mixed finite element methods based on $C^{0}$ spaces are therefore widely used, with the chemical potential introduced as an additional unknown to reduce the fourth-order equation to a coupled system of second-order equations \cite{KayStylesWelford2008, BrennerScott2008, Gunzburger1989}. Rigorous error estimates and adaptive mesh-refinement strategies have been developed for such mixed formulations \cite{DiegelFengWise2016, Hintermuller2013}. Alternatively, Discontinuous Galerkin (DG) methods provide local approximation flexibility and are well suited to adaptive discretizations \cite{LiuRiviere2020, GuoXu2019, Manzanero2020, GiesselmannPryer2015}. Structured-grid approaches have also been extensively studied, including mass-conservative finite-difference methods on staggered MAC grids \cite{Guo2017, Balashov2020}, Lattice Boltzmann methods \cite{Liang2014, Zheng2015}, and high-order spectral solvers \cite{Roccon2025}.

Temporally, the development of energy-stable methods has driven substantial algorithmic progress. Fully implicit energy-stable schemes can retain an appropriate discrete energy-dissipation structure but generally require solving large coupled nonlinear systems \cite{Feng2006, KIM2004}. Convex-splitting approaches provide unconditional energy stability for phase-field equations \cite{Eyre1998, HeLiuTang2007}, although their implicit treatment of nonlinear terms can require nonlinear iterative solvers at each time step \cite{Layton2008, LaytonTrenchea2012}. Researchers subsequently developed stabilized semi-implicit schemes to avoid nonlinear iterations by introducing additional stabilization terms \cite{ShenYang2010, WangYu2018, Du2018}. Such schemes yield linear systems, although the choice and magnitude of the stabilization parameter can influence numerical dissipation and accuracy. Another important class of approaches is based on Invariant Energy Quadratization (IEQ) \cite{YANG2017} and Scalar Auxiliary Variable (SAV) formulations \cite{ShenXuYang2018}, which rewrite nonlinear free-energy contributions in quadratic form through auxiliary variables. In CHNS applications, SAV and multiple-SAV formulations can lead to linear, fully decoupled schemes that require only constant-coefficient elliptic solves at each time step \cite{LiShen2022}. The corresponding stability laws are typically formulated in terms of an auxiliary or modified discrete energy rather than solely in terms of the original phase and velocity variables \cite{LiShen2022_MSAV}. For the incompressible fluid subsystem, spatial discretizations commonly employ inf-sup stable velocity-pressure pairs, while fractional-step and projection methods are widely used to reduce the coupling between velocity and pressure \cite{Chorin1968, Guermond2006}. Researchers have since developed various stabilization, penalty, projection, and splitting strategies to improve the accuracy and robustness of these temporal discretizations \cite{Decaria2017, Hurl2014, JiangLayton2015, JiangTran2015, JiangMohebujjaman2016}. Concurrently, data-driven and physics-informed computational frameworks have emerged for phase-field and multiphase-flow problems \cite{Qiu2025, Mattey2022, Zhu2023, Bamdad2024}, although the robust enforcement of conservation laws and discrete physical structure remains an active area of research.

The numerical simulation of complex, highly nonlinear fluid dynamics can impose severe restrictions on the time-step size required for stable and accurate computations, particularly for strongly coupled or under-resolved flows. One strategy for improving robustness without reducing the formal temporal order is to introduce discrete regularization based on the temporal curvature of the numerical solution. Jiang et al. \cite{JiangMohebujjaman2016} developed an optimally accurate discrete regularization for second-order time-stepping methods for the Navier-Stokes equations. By incorporating stabilization terms proportional to the discrete temporal curvature of the solution, they obtained a family of second-order time-stepping methods designed to improve stability while retaining second-order temporal accuracy. Building on this framework, Jiang and Yang \cite{JiangYang2024} extended curvature-based regularization to ensemble simulations of incompressible flows using a regularization parameter $\epsilon>0$. Their analysis and numerical experiments show that the regularization can substantially improve the robustness of the corresponding flow computations, particularly in regimes where weakly dissipative or unregularized time discretizations exhibit oscillatory behavior. Motivated by these developments, the proposed linear IMEX scheme incorporates an analogous temporal-curvature regularization mechanism, controlled by $\epsilon\geq0$, for the Cahn-Hilliard-Navier-Stokes system.

Despite substantial progress in second-order, energy-stable, linear, and decoupled algorithms for CHNS-type systems \cite{HanWang2017, LiShen2022}, interest remains in simple time-discretization frameworks that combine second-order temporal accuracy, unconditional stability, linear treatment of nonlinear couplings, and a tunable regularization mechanism within a unified formulation. In this article, we propose a family of second-order, linear, unconditionally stable methods for solving the CHNS equations modeling two-phase flows of matched density.

Our primary contribution is a fully linear semi-discrete algorithm that treats the stiff nonlinear couplings explicitly through extrapolated quantities while incorporating temporal-curvature regularization via the parameter $\epsilon$. By establishing a set of symmetric positive definite matrices, we define the corresponding discrete norms and present a mathematical proof that the proposed scheme satisfies a discrete energy-dissipation law without a time-step restriction. The resulting formulation requires only solving linear systems at each time step and provides a flexible family of second-order methods parameterized by the temporal discretization and regularization parameters.

The remainder of this article is organized as follows. In Section \ref{model}, we introduce the continuous CHNS model and present the proposed semi-discrete numerical scheme. Section \ref{uncon_stab} introduces discrete matrix norms and establishes the scheme's unconditional stability. Finally, Section \ref{nume} presents numerical experiments examining temporal convergence, energy behavior, numerical robustness, and representative multiphase-flow benchmarks.

    \section{The model and the numerical scheme}\label{model}
	
	\subsection{The model}
	As an example we present our numerical algorithm for solving the Cahn-Hilliard-Navier-Stokes system that models two-phase flows of matched density. Similar schemes can be constructed for other phase field fluid models.
	We consider a mixture of two immiscible, incompressible fluids in a bounded Lipschitz domain $\Omega$ in $\mathbb{R}^d$ $(d=2, 3)$ with matched density assumed to be unity for simplicity of presentation. Introducing a phase function $\phi$ such that
	\begin{equation}\label{phase-function}
		\phi(x,t)\approx \left\{\begin{aligned}
			&1, \qquad  \text{for fluid 1},\\
			&-1, \quad \text{for fluid 2},
		\end{aligned}\right.
	\end{equation}
	we adopt the Ginzburg-Landau type free energy associated with the binary system
	\begin{align}\label{free_en}
		W(\phi, \nabla \phi)=\int_{\Omega} \lambda \left( \frac{1}{2}\vert \nabla \phi \vert^2+F(\phi)\right)\, dx.
	\end{align}
	In \eqref{free_en} the first term contributes to the hydrophilic type (tendency of mixing) of interactions between the materials while the second part, the double-well bulk energy $F(\phi)=\frac{1}{4\eta^2}(\phi^2-1)^2$, represents the hydrophobic type (tendency of separation) of interactions. As the consequence of the competition between the two types of interactions, the equilibrium configuration will include a diffusive interface with thickness proportional to the parameter $\eta$.

	The governing equations are the following  Cahn-Hilliard-Navier-Stokes (CHNS) system:
	\begin{equation}\label{eq:CHNS}
		\left\{\begin{aligned}
			& \phi_{t}+\nabla \cdot{(\phi\bfu)}  =M \Delta \mu,\\
			&\mu =\lambda (-\Delta \phi  + f(\phi)),\\
			&\bfu_t+ \bfu\cdot \nabla \bfu + \nabla p -\nu \Delta \bfu=-\phi \nabla \mu,\\
			&\nabla \cdot \bfu = 0,
		\end{aligned}\right.
	\end{equation}
	equipped with the boundary conditions
	$$\bfu|_{\partial \Omega}=0, \quad \partial_{n}\phi |_{\partial\Omega}=0, \quad \partial_n\mu|_{\partial\Omega}=0.$$
	Assuming no external forcing other than gravity, the CHNS system satisfies an energy law, i.e.
	\begin{align}\label{en_law}
		\frac{d}{dt}\left\{W(\phi, \nabla\phi)+\int_{\Omega}\frac{1}{2}|\mathbf{u}|^2dx\right\}=-\int_{\Omega}M|\nabla \mu|^2+\nu |\nabla \mathbf{u}|^2 dx. 
	\end{align}

	\subsection{The methods}
	In this section we introduce the algorithm in the semi-discrete form and establish its unconditional long-time stability.
	Throughout, the $L^2(\Omega)$ norm of scalars, vectors, and tensors will be denoted by $\Vert \cdot\Vert$ with the usual $L^2$ inner product denoted by $(\cdot, \cdot)$.

	Denote the interpolation at $t=t_{n+\theta}$ by 
	$${\cal J}_{n+\theta}^{\epsilon}(a):=\theta \frac{\nu+\epsilon}{\nu}a_{n+1}+(1-\theta\frac{\nu+2\epsilon}{\nu})a_n+\theta \frac{\epsilon}{\nu} a_{n-1},$$
	and extrapolation at $t=t_{n+\theta}$ by
	$${\cal H}_{n+\theta}(a):=(\theta+1) a_n-\theta a_{n-1},$$ from \cite{JiangYang2024}.
	
	Let $q=\frac{1}{\eta^2}(\phi^2-1)$ and thus $f(\phi)=\phi q$. Taking the derivative of $q$ with respect to time gives $q_t=\frac{2}{\eta^2}\phi\phi_t$, which can be discretized as follows.
	$$\frac{(\theta+\frac{1}{2})q^{n+1}-2\theta q^n+(\theta - \frac{1}{2})q^{n-1}}{\Delta t}=\frac{2}{\eta^2} {\cal H}_{n+\theta}(\phi) \frac{(\theta+\frac{1}{2})\phi^{n+1}-2\theta \phi^n+(\theta - \frac{1}{2})\phi^{n-1}}{\Delta t},$$
	
	\noindent where $\theta\in \left( \frac{1}{2} \text{, } 1 \right], \epsilon \geq 0$.
	
	We then propose a family of second-order,  linear, unconditionally stable methods for the CHNS system given by
	
	\begin{Aalgorithm}\label{Algo-CHNS}
		Given $\bfu^{n-1}$, $\bfu^n$, $p^{n-1}$, $p^n$, $\phi^{n-1}$, $\phi^n$, $q^{n-1}$, $q^n$, and $\mu^{n-1}$, find $\bfu^{n+1}$, $p^{n+1}$, $\phi^{n+1}$, $q^{n+1}$, and $\mu^n$ satisfying
		\begin{align}
			&\frac{(\theta+\frac{1}{2})\phi^{n+1}-2\theta \phi^n+(\theta - \frac{1}{2})\phi^{n-1}}{\Delta t}+\nabla \cdot \left({\cal H}_{n+\theta}(\phi) {\cal J}^{\epsilon}_{n+\theta}(u) \right)  -M \Delta {\cal H}_{n+\theta}(\mu)=0,\label{CHNS1}\\
			& {\cal H}_{n+\theta}(\mu)=\lambda\left(- \Delta {\cal J}^{\epsilon}_{n+\theta}(\phi)+{\cal H}_{n+\theta}(\phi) {\cal J}^{\epsilon}_{n+\theta}(q) \right),\label{CHNS2}\\
			&
			\frac{(\theta+\frac{1}{2})q^{n+1}-2\theta q^n+(\theta - \frac{1}{2})q^{n-1}}{\Delta t}=\frac{2}{\eta^2} {\cal H}_{n+\theta}(\phi) \frac{(\theta+\frac{1}{2})\phi^{n+1}-2\theta \phi^n+(\theta - \frac{1}{2})\phi^{n-1}}{\Delta t}, \label{CHNS2a}\\
			&
			\frac{(\theta+\frac{1}{2})u^{n+1}-2\theta u^n+(\theta - \frac{1}{2})u^{n-1}}{\Delta t}+{\cal H}_{n+\theta}(u) \cdot \nabla \left({\cal J}^{\epsilon}_{n+\theta}(u) \right)
			\nonumber\\
			&-\nu\Delta  {\cal J}^{\epsilon}_{n+\theta}(u)+\nabla {\cal J}^{\epsilon}_{n+\theta}(p) +{\cal H}_{n+\theta}(\phi)\nabla{\cal H}_{n+\theta}(\mu) =0, 
			\label{CHNS3}\\
			&\nabla \cdot \bfu^{n+1}=0,\label{CHNS4} \\
			&\bfu^{n+1}|_{\partial \Omega}=0, \quad \nabla \phi^{n+1}\cdot n|_{\partial \Omega}=0, \quad  \nabla \mu^{n}\cdot n|_{\partial \Omega}=0. \label{CHNS5}
		\end{align}
		
	\end{Aalgorithm}
	
	\section{Unconditional Stability}\label{uncon_stab}
	
	Define the symmetric positive definite matrix $F\in \mathbb{R} ^{n\times n}$ by
	\begin{equation*}
		F=\theta (2\theta -1) I + \frac{4\theta^2\epsilon}{\nu} I,
	\end{equation*}
	and the symmetric matrix $G \in \mathbb{R}^{2n\times 2n}$ as follows.
	\[ G=\left(\begin{matrix}
		\frac{\theta(2\theta +3)}{4}\frac{\nu+\epsilon}{\nu} I- \frac{\theta(2\theta+1)}{4}\frac{\epsilon}{\nu} I  & -(\frac{(\theta+1)(2\theta-1)}{4}\frac{\nu+\epsilon}{\nu}I +\frac{(1-\theta)(2\theta+1)}{4}\frac{\epsilon}{\nu}I)  
		\vspace{.2cm}\\
		-(\frac{(\theta+1)(2\theta-1)}{4}\frac{\nu+\epsilon}{\nu}I+\frac{(1-\theta)(2\theta+1)}{4}\frac{\epsilon}{\nu} I) &  \frac{\theta(2\theta -1)}{4}\frac{\nu+\epsilon}{\nu} I+ \frac{\theta(-2\theta+3)}{4}\frac{\epsilon}{\nu} I  
	\end{matrix} \right).
	\]
	For any ${\bf u}, {\bf v} \in \mathbb{R}^{n}$, define $F-$norm of the $n$ vector ${\bf u}$:
	\begin{equation*}
		\|{\bf u}\|_F^2=\left( {\bf u}, F{\bf u} \right),
	\end{equation*}
	\noindent which is non-negative, and $G-$norm of the $2n$ vector $\left[ \begin{matrix} {\bf u}\cr {\bf v} \end{matrix} \right]$:
	\begin{equation*}
		{\left\| \left[ \begin{matrix} {\bf u}\cr {\bf v} \end{matrix} \right]\right\|} ^2_G=\left(\left[ \begin{matrix} {\bf u}\cr {\bf v} \end{matrix} \right], G\left[ \begin{matrix} {\bf u}\cr {\bf v} \end{matrix} \right]\right) ,
	\end{equation*}
	whose value could be negative.
	\begin{lemma} \label{lm:G-form}
		For any vector ${\bf u}$ , ${\bf v} \in \mathbb{R}^n$, we have, \cite{JiangMohebujjaman2016},
		\begin{align}\label{G1}
			\Big(\Big[ \begin{matrix} {\bf u}\cr {\bf v} \end{matrix} \Big], G\Big[ \begin{matrix} {\bf u}\cr {\bf v} \end{matrix} \Big]\Big) &=
			\frac{2\theta +1}{4}  \Vert {\bf u}\Vert^2 + \frac{-2\theta +1}{4} \Vert {\bf v} \Vert^2
			\\
			&\quad
			+\frac{(\theta+1)(2\theta-1)}{\color{black} 4}\Vert {\bf u}-{\bf v}\Vert^2 
			+\frac{\theta}{2}\frac{\epsilon}{\nu}\Vert {\bf u}-{\bf v}\Vert^2
			\nonumber
			\\
			&
			\geq \frac{2\theta+1}{4}\|{\bf u}\|^2 - \frac{2\theta -1}{4} \| {\bf v}\|^2 \text{,} \nonumber
		\end{align}
		and
		\begin{align}\label{G2}
			\left(\left[ \begin{matrix} {\bf u}\cr {\bf v} \end{matrix} \right], G\Big[ \begin{matrix} {\bf u}\cr {\bf v} \end{matrix} \Big]\right)
			\leq \frac{2\theta+1}{4}\|{\bf u}\|^2 +\frac{(\theta+1)(2\theta-1)}{{\color{black}4}}\Vert {\bf u}-{\bf v}\Vert^2 
			+\frac{\theta}{2}\frac{\epsilon}{\nu}\Vert {\bf u}-{\bf v}\Vert^2\\
			\leq   
			\left( \frac{2\theta+1}{4}+\frac{(\theta+1)(2\theta-1)}{2} +\frac{\theta\epsilon}{\nu}\right) \Vert {\bf u}\Vert^2
			+\left(\frac{(\theta+1)(2\theta-1)}{2} +\frac{\theta\epsilon}{\nu}\right)\Vert {\bf v}\Vert^2\text{.} \nonumber
		\end{align}
	\end{lemma}

	\begin{theorem}\label{main-th}
		The method \eqref{CHNS1}--\eqref{CHNS5} is unconditionally long-time stable in the sense that for any $N\geq 2$
		\begin{align*}
			&
			\lambda \Vert \nabla \phi_{N}\Vert^2
			+\frac{\lambda\eta^2}{2} \Vert q_{N}\Vert^2
			+\Vert u_{N}\Vert^2
			\\
			&
			+\frac{\lambda}{2\theta+1} \sum_{n=1}^{N-1}\| \nabla(\phi_{n+1}- 2 \phi_n+\phi_{n-1})\|^2_F
			+\frac{\lambda\eta^2}{2(2\theta+1)}\sum_{n=1}^{N-1} \| q_{n+1}- 2 q_n+q_{n-1}\|^2_F
			\nonumber\\
			&
			+\frac{1}{2\theta+1} \sum_{n=1}^{N-1}\| u_{n+1}- 2 u_n+u_{n-1}\|^2_F
			+\frac{4 \Delta t\nu}{2\theta+1} \sum_{n=1}^{N-1}\Vert\nabla {\cal J}^{\epsilon}_{n+\theta}(u)\Vert^2
			+\frac{4\Delta t M }{2\theta+1}\sum_{n=1}^{N-1} \Vert \nabla {\cal H}_{n+\theta}(\mu)\Vert^2
			\nonumber\\
			&
			\leq \left(\frac{2\theta  -1 }{2\theta + 1}\right)^N \left( \lambda \Vert \nabla \phi_{0}\Vert^2
			+\frac{\lambda\eta^2}{2} \Vert q_{0}\Vert^2
			+\Vert u_{0}\Vert^2 \right)
			\nonumber\\
			&\qquad
			+2(1-\left(\frac{2\theta  -1 }{2\theta + 1}\right)^N) \left(\lambda {\left\| \left[ \begin{matrix} \nabla \phi_{1}\cr \nabla \phi_{0} \end{matrix} \right]\right\|} ^2_G
			+\frac{\lambda\eta^2}{2} {\left\| \left[ \begin{matrix} q_{1}\cr q_{0} \end{matrix} \right]\right\|} ^2_G
			+{\left\| \left[ \begin{matrix} u_{1}\cr u_{0} \end{matrix} \right]\right\|} ^2_G \right). \nonumber
		\end{align*}
	\end{theorem}
	\begin{proof}
		Taking the inner product of \eqref{CHNS1} with ${\cal H}_{n+\theta}(\mu)$ and multiplying through by $\Delta t$ gives
		\begin{align}
			&\left((\theta+\frac{1}{2})\phi^{n+1}-2\theta \phi^n+(\theta - \frac{1}{2})\phi^{n-1}, {\cal H}_{n+\theta}(\mu)\right)
			-\Delta t\left({\cal H}_{n+\theta}(\phi){\cal J}^{\epsilon}_{n+\theta}(u), \nabla {\cal H}_{n+\theta}(\mu)\right)
			\nonumber\\
			&\qquad
			+\Delta t M \Vert \nabla {\cal H}_{n+\theta}(\mu)\Vert^2=0.\label{cpf1}
		\end{align}

		Taking the inner product of \eqref{CHNS2} with $(\theta+\frac{1}{2})\phi^{n+1}-2\theta \phi^n+(\theta - \frac{1}{2})\phi^{n-1}$ and integrating by parts gives
		\begin{align}\label{cpf2}
			&-\left((\theta+\frac{1}{2})\phi^{n+1}-2\theta \phi^n+(\theta - \frac{1}{2})\phi^{n-1},{\cal H}_{n+\theta}(\mu)\right)
			\\
			&+\lambda \left( {\left\| \left[ \begin{matrix} \nabla \phi_{n+1}\cr \nabla \phi_n \end{matrix} \right]\right\|} ^2_G - {\left\| \left[ \begin{matrix} \nabla \phi_{n}\cr \nabla \phi_{n-1} \end{matrix} \right]\right\|} ^2_G
			+\frac{1}{4} \| \nabla(\phi_{n+1}- 2 \phi_n+\phi_{n-1})\|^2_F \right)
			\nonumber \\
			&+\lambda \left( {\cal H}_{n+\theta}(\phi) {\cal J}^{\epsilon}_{n+\theta}(q), (\theta+\frac{1}{2})\phi^{n+1}-2\theta \phi^n+(\theta - \frac{1}{2})\phi^{n-1} \right) =0.\nonumber
		\end{align}

		Taking the inner product of \eqref{CHNS2a} with $\frac{\lambda\eta^2}{2}{\cal J}^{\epsilon}_{n+\theta}(q)$ and multiplying through by $\Delta t$ gives
		\begin{align}
			&
			\frac{\lambda\eta^2}{2} \left({\left\| \left[ \begin{matrix} q_{n+1}\cr q_n \end{matrix} \right]\right\|} ^2_G - {\left\| \left[ \begin{matrix} q_{n}\cr q_{n-1} \end{matrix} \right]\right\|} ^2_G
			+\frac{1}{4} \| q_{n+1}- 2 q_n+q_{n-1}\|^2_F \right)
			\nonumber\\
			&\qquad
			= \lambda \left( {\cal H}_{n+\theta}(\phi) \left((\theta+\frac{1}{2})\phi^{n+1}-2\theta \phi^n+(\theta - \frac{1}{2})\phi^{n-1}\right), {\cal J}^{\epsilon}_{n+\theta}(q)\right).\label{cpf2a}
		\end{align}

		Summing up \eqref{cpf1}, \eqref{cpf2} and \eqref{cpf2a} yields
		\begin{align}\label{cpf20}
			&-\Delta t\left({\cal H}_{n+\theta}(\phi){\cal J}^{\epsilon}_{n+\theta}(u), \nabla {\cal H}_{n+\theta}(\mu)\right)
			+\Delta t M \Vert \nabla {\cal H}_{n+\theta}(\mu)\Vert^2 \nonumber \\
			&+\lambda \left( {\left\| \left[ \begin{matrix} \nabla \phi_{n+1}\cr \nabla \phi_n \end{matrix} \right]\right\|} ^2_G - {\left\| \left[ \begin{matrix} \nabla \phi_{n}\cr \nabla \phi_{n-1} \end{matrix} \right]\right\|} ^2_G
			+\frac{1}{4} \| \nabla(\phi_{n+1}- 2 \phi_n+\phi_{n-1})\|^2_F \right) \\
			&+\frac{\lambda\eta^2}{2} \left({\left\| \left[ \begin{matrix} q_{n+1}\cr q_n \end{matrix} \right]\right\|} ^2_G - {\left\| \left[ \begin{matrix} q_{n}\cr q_{n-1} \end{matrix} \right]\right\|} ^2_G
			+\frac{1}{4} \| q_{n+1}- 2 q_n+q_{n-1}\|^2_F \right)=0. \nonumber
		\end{align}

		Taking the inner product of \eqref{CHNS3} with $\Delta t {\cal J}^{\epsilon}_{n+\theta}(u)$ gives
		\begin{align}
			&{\left\| \left[ \begin{matrix} u_{n+1}\cr u_n \end{matrix} \right]\right\|} ^2_G - {\left\| \left[ \begin{matrix} u_{n}\cr u_{n-1} \end{matrix} \right]\right\|} ^2_G
			+\frac{1}{4} \| u_{n+1}- 2 u_n+u_{n-1}\|^2_F \label{cpf3}\\
			&+\Delta t\nu\Vert\nabla {\cal J}^{\epsilon}_{n+\theta}(u)\Vert^2
			+\Delta t\left( {\cal H}_{n+\theta}(\phi)\nabla{\cal H}_{n+\theta}(\mu), {\cal J}^{\epsilon}_{n+\theta}(u)\right)=0.
		\end{align}

		The convection and pressure terms vanish by incompressibility and the homogeneous velocity boundary condition. Adding \eqref{cpf20} and \eqref{cpf3} yields
		\begin{align}\label{cpf6}
			&
			\Delta t M \Vert \nabla {\cal H}_{n+\theta}(\mu)\Vert^2
			+\lambda \left( {\left\| \left[ \begin{matrix} \nabla \phi_{n+1}\cr \nabla \phi_n \end{matrix} \right]\right\|} ^2_G - {\left\| \left[ \begin{matrix} \nabla \phi_{n}\cr \nabla \phi_{n-1} \end{matrix} \right]\right\|} ^2_G
			+\frac{1}{4} \| \nabla(\phi_{n+1}- 2 \phi_n+\phi_{n-1})\|^2_F \right) \\
			&+\frac{\lambda\eta^2}{2} \left({\left\| \left[ \begin{matrix} q_{n+1}\cr q_n \end{matrix} \right]\right\|} ^2_G - {\left\| \left[ \begin{matrix} q_{n}\cr q_{n-1} \end{matrix} \right]\right\|} ^2_G
			+\frac{1}{4} \| q_{n+1}- 2 q_n+q_{n-1}\|^2_F \right)
			\nonumber\\
			&
			+{\left\| \left[ \begin{matrix} u_{n+1}\cr u_n \end{matrix} \right]\right\|} ^2_G - {\left\| \left[ \begin{matrix} u_{n}\cr u_{n-1} \end{matrix} \right]\right\|} ^2_G
			+\frac{1}{4} \| u_{n+1}- 2 u_n+u_{n-1}\|^2_F
			+\Delta t\nu\Vert\nabla {\cal J}^{\epsilon}_{n+\theta}(u)\Vert^2
			=0. \nonumber
		\end{align}

		Taking the sum from $n=1$ to $n=N-1$ yields
		\begin{align}\label{cpf7}
			&
			\lambda {\left\| \left[ \begin{matrix} \nabla \phi_{N}\cr \nabla \phi_{N-1} \end{matrix} \right]\right\|} ^2_G
			+\frac{\lambda\eta^2}{2} {\left\| \left[ \begin{matrix} q_{N}\cr q_{N-1} \end{matrix} \right]\right\|} ^2_G
			+{\left\| \left[ \begin{matrix} u_{N}\cr u_{N-1} \end{matrix} \right]\right\|} ^2_G
			\\
			&
			+\frac{\lambda}{4} \sum_{n=1}^{N-1}\| \nabla(\phi_{n+1}- 2 \phi_n+\phi_{n-1})\|^2_F
			+\frac{\lambda\eta^2}{8}\sum_{n=1}^{N-1} \| q_{n+1}- 2 q_n+q_{n-1}\|^2_F
			\nonumber\\
			&
			+\frac{1}{4} \sum_{n=1}^{N-1}\| u_{n+1}- 2 u_n+u_{n-1}\|^2_F
			+\Delta t\nu \sum_{n=1}^{N-1}\Vert\nabla {\cal J}^{\epsilon}_{n+\theta}(u)\Vert^2
			+\Delta t M \sum_{n=1}^{N-1} \Vert \nabla {\cal H}_{n+\theta}(\mu)\Vert^2
			\nonumber\\
			&
			= \lambda {\left\| \left[ \begin{matrix} \nabla \phi_{1}\cr \nabla \phi_{0} \end{matrix} \right]\right\|} ^2_G
			+\frac{\lambda\eta^2}{2} {\left\| \left[ \begin{matrix} q_{1}\cr q_{0} \end{matrix} \right]\right\|} ^2_G
			+ {\left\| \left[ \begin{matrix} u_{1}\cr u_{0} \end{matrix} \right]\right\|} ^2_G . \nonumber
		\end{align}

		By Lemma \ref{lm:G-form}, we have
		\begin{align}\label{cpf8}
			&
			\lambda \Vert \nabla \phi_{N}\Vert^2
			+\frac{\lambda\eta^2}{2} \Vert q_{N}\Vert^2
			+\Vert u_{N}\Vert^2
			\\
			&
			+\frac{\lambda}{2\theta+1} \sum_{n=1}^{N-1}\| \nabla(\phi_{n+1}- 2 \phi_n+\phi_{n-1})\|^2_F
			+\frac{\lambda\eta^2}{2(2\theta+1)}\sum_{n=1}^{N-1} \| q_{n+1}- 2 q_n+q_{n-1}\|^2_F
			\nonumber\\
			&
			+\frac{1}{2\theta+1} \sum_{n=1}^{N-1}\| u_{n+1}- 2 u_n+u_{n-1}\|^2_F
			+\frac{4 \Delta t\nu}{2\theta+1} \sum_{n=1}^{N-1}\Vert\nabla {\cal J}^{\epsilon}_{n+\theta}(u)\Vert^2
			+\frac{4\Delta t M }{2\theta+1}\sum_{n=1}^{N-1} \Vert \nabla {\cal H}_{n+\theta}(\mu)\Vert^2
			\nonumber\\
			&
			\leq \lambda \frac{2\theta  -1 }{2\theta + 1} \Vert \nabla \phi_{N-1}\Vert^2
			+\frac{\lambda\eta^2}{2} \frac{2\theta  -1 }{2\theta + 1} \Vert q_{N-1}\Vert^2
			+\frac{2\theta  -1 }{2\theta + 1}\Vert u_{N-1}\Vert^2
			\nonumber\\
			&\qquad
			+\lambda \frac{4}{2\theta+1} {\left\| \left[ \begin{matrix} \nabla \phi_{1}\cr \nabla \phi_{0} \end{matrix} \right]\right\|} ^2_G
			+\frac{\lambda\eta^2}{2} \frac{4}{2\theta+1} {\left\| \left[ \begin{matrix} q_{1}\cr q_{0} \end{matrix} \right]\right\|} ^2_G
			+ \frac{4}{2\theta+1} {\left\| \left[ \begin{matrix} u_{1}\cr u_{0} \end{matrix} \right]\right\|} ^2_G . \nonumber
		\end{align}

		Applying the same lower bound in Lemma \ref{lm:G-form} to the initial pair gives the corresponding base estimate for $N=1$. By induction, we obtain the following energy estimate
		\begin{align}
			&
			\lambda \Vert \nabla \phi_{N}\Vert^2
			+\frac{\lambda\eta^2}{2} \Vert q_{N}\Vert^2
			+\Vert u_{N}\Vert^2
			\\
			&
			+\frac{\lambda}{2\theta+1} \sum_{n=1}^{N-1}\| \nabla(\phi_{n+1}- 2 \phi_n+\phi_{n-1})\|^2_F
			+\frac{\lambda\eta^2}{2(2\theta+1)}\sum_{n=1}^{N-1} \| q_{n+1}- 2 q_n+q_{n-1}\|^2_F
			\nonumber\\
			&
			+\frac{1}{2\theta+1} \sum_{n=1}^{N-1}\| u_{n+1}- 2 u_n+u_{n-1}\|^2_F
			+\frac{4 \Delta t\nu}{2\theta+1} \sum_{n=1}^{N-1}\Vert\nabla {\cal J}^{\epsilon}_{n+\theta}(u)\Vert^2
			+\frac{4\Delta t M }{2\theta+1}\sum_{n=1}^{N-1} \Vert \nabla {\cal H}_{n+\theta}(\mu)\Vert^2
			\nonumber\\
			&
			\leq \left(\frac{2\theta  -1 }{2\theta + 1}\right)^N \left( \lambda \Vert \nabla \phi_{0}\Vert^2
			+\frac{\lambda\eta^2}{2} \Vert q_{0}\Vert^2
			+\Vert u_{0}\Vert^2 \right)
			\nonumber\\
			&\qquad
			+2(1-\left(\frac{2\theta  -1 }{2\theta + 1}\right)^N) \left(\lambda {\left\| \left[ \begin{matrix} \nabla \phi_{1}\cr \nabla \phi_{0} \end{matrix} \right]\right\|} ^2_G
			+\frac{\lambda\eta^2}{2} {\left\| \left[ \begin{matrix} q_{1}\cr q_{0} \end{matrix} \right]\right\|} ^2_G
			+{\left\| \left[ \begin{matrix} u_{1}\cr u_{0} \end{matrix} \right]\right\|} ^2_G \right). \nonumber
		\end{align}
		This completes the proof.
	\end{proof}

\section{Numerical experiments}\label{nume}

We present a series of benchmark computations to examine the accuracy, robustness, conservation properties, and practical performance of the proposed algorithm. Several of the tests are adapted from commonly used benchmarks in the literature \cite{Han2015, Gao2025, Lee2012}. Unless otherwise stated, the spatial discretization uses the $P_2-P_1$ Taylor--Hood pair for $(\mathbf{u},p)$ and $P_2-P_2$ elements for $(\phi,\mathcal{H}(\mu))$.

\subsection{Convergence Analysis: Temporal Accuracy}

To examine the convergence behavior of the scheme, we employ the Method of Manufactured Solutions (MMS) on the domain $\Omega=[0,1]^2$. Artificial forcing terms are added to the continuous Cahn--Hilliard--Navier--Stokes equations so that the following functions constitute an exact analytical solution:
\begin{align}
    u_{ex}(x,y,t) &=
    \begin{pmatrix}
    \pi \sin^2(\pi x)\sin(2\pi y)\cos(t) \\
    -\pi\sin(2\pi x)\sin^2(\pi y)\cos(t)
    \end{pmatrix},
    \label{eq:mms_u}\\
    p_{ex}(x,y,t) &= \cos(\pi x)\cos(\pi y)\cos(t),
    \label{eq:mms_p}\\
    \phi_{ex}(x,y,t) &= 0.1\cos(t)\cos(\pi x)\cos(\pi y).
    \label{eq:mms_phi}
\end{align}
Initial data and boundary data consistent with these exact profiles are imposed. Across all convergence tests, the simulations are advanced to $T=1.0$, with
\[
\theta=0.8,\qquad
\nu=1.0,\qquad
\lambda=0.1,\qquad
\eta=1.0,\qquad
M=1.0,\qquad
\epsilon=10^{-5}.
\]
We consider two mixed finite-element configurations: $P_2-P_1-P_2-P_2$ (degree $k=2$) and the higher-order $P_3-P_2-P_3-P_3$ configuration (degree $k=3$).

To examine the temporal accuracy of the scheme, we fix a highly refined spatial mesh with $N_x=128$ and successively halve the time step from $\Delta t=0.1$ to $\Delta t=0.0125$. Tables \ref{tab:temp_conv_deg2} and \ref{tab:temp_conv_deg3} report the discrete $L^2$ errors for the velocity, pressure, and phase field. For the current computations, the observed rates are close to second order for all three variables.

\begin{table}[h!]
\centering
\caption{Temporal convergence at $T=1.0$ utilizing $P_2-P_1-P_2-P_2$ elements (fixed $N_x=128$).}
\label{tab:temp_conv_deg2}
\begin{tabular}{ccccccc}
\toprule
$\Delta t$ & $\|u-u_{ex}\|_{L^2}$ & Rate & $\|p-p_{ex}\|_{L^2}$ & Rate & $\|\phi-\phi_{ex}\|_{L^2}$ & Rate \\
\hline
$1.00 \times 10^{-1}$ & $9.16 \times 10^{-4}$ & --   & $1.41 \times 10^{-2}$ & --   & $1.46 \times 10^{-3}$ & --   \\
$5.00 \times 10^{-2}$ & $2.22 \times 10^{-4}$ & 2.04 & $3.60 \times 10^{-3}$ & 1.97 & $3.48 \times 10^{-4}$ & 2.07 \\
$2.50 \times 10^{-2}$ & $5.49 \times 10^{-5}$ & 2.02 & $9.08 \times 10^{-4}$ & 1.99 & $8.51 \times 10^{-5}$ & 2.03 \\
$1.25 \times 10^{-2}$ & $1.37 \times 10^{-5}$ & 2.00 & $2.33 \times 10^{-4}$ & 1.96 & $2.10 \times 10^{-5}$ & 2.01 \\
\bottomrule
\end{tabular}
\end{table}

\begin{table}[h!]
\centering
\caption{Temporal convergence at $T=1.0$ utilizing $P_3-P_2-P_3-P_3$ elements (fixed $N_x=128$).}
\label{tab:temp_conv_deg3}
\begin{tabular}{ccccccc}
\toprule
$\Delta t$ & $\|u-u_{ex}\|_{L^2}$ & Rate & $\|p-p_{ex}\|_{L^2}$ & Rate & $\|\phi-\phi_{ex}\|_{L^2}$ & Rate \\
\hline
$1.00 \times 10^{-1}$ & $9.16 \times 10^{-4}$ & --   & $1.41 \times 10^{-2}$ & --   & $1.46 \times 10^{-3}$ & --   \\
$5.00 \times 10^{-2}$ & $2.22 \times 10^{-4}$ & 2.04 & $3.60 \times 10^{-3}$ & 1.97 & $3.48 \times 10^{-4}$ & 2.07 \\
$2.50 \times 10^{-2}$ & $5.49 \times 10^{-5}$ & 2.02 & $9.06 \times 10^{-4}$ & 1.99 & $8.51 \times 10^{-5}$ & 2.03 \\
$1.25 \times 10^{-2}$ & $1.36 \times 10^{-5}$ & 2.01 & $2.27 \times 10^{-4}$ & 1.99 & $2.11 \times 10^{-5}$ & 2.01 \\
\bottomrule
\end{tabular}
\end{table}

\begin{figure}[h!]
    \centering
    \begin{minipage}{0.48\textwidth}
        \centering
        \includegraphics[width=\linewidth]{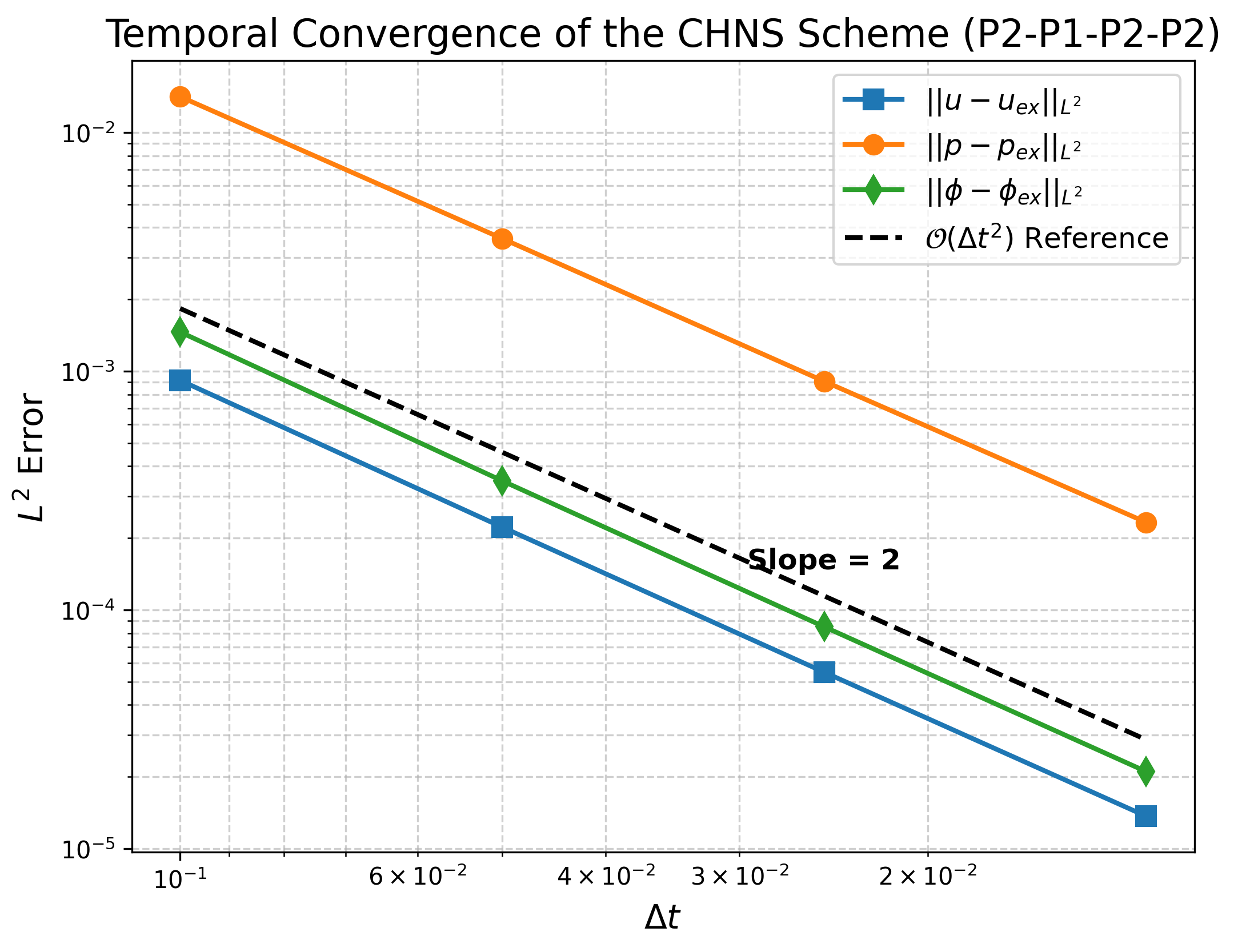}
        \caption{Temporal error decay ($P_2$-based) consistent with approximately second-order temporal accuracy over the tested range.}
        \label{fig:temp_conv_deg2}
    \end{minipage}\hfill
    \begin{minipage}{0.48\textwidth}
        \centering
        \includegraphics[width=\linewidth]{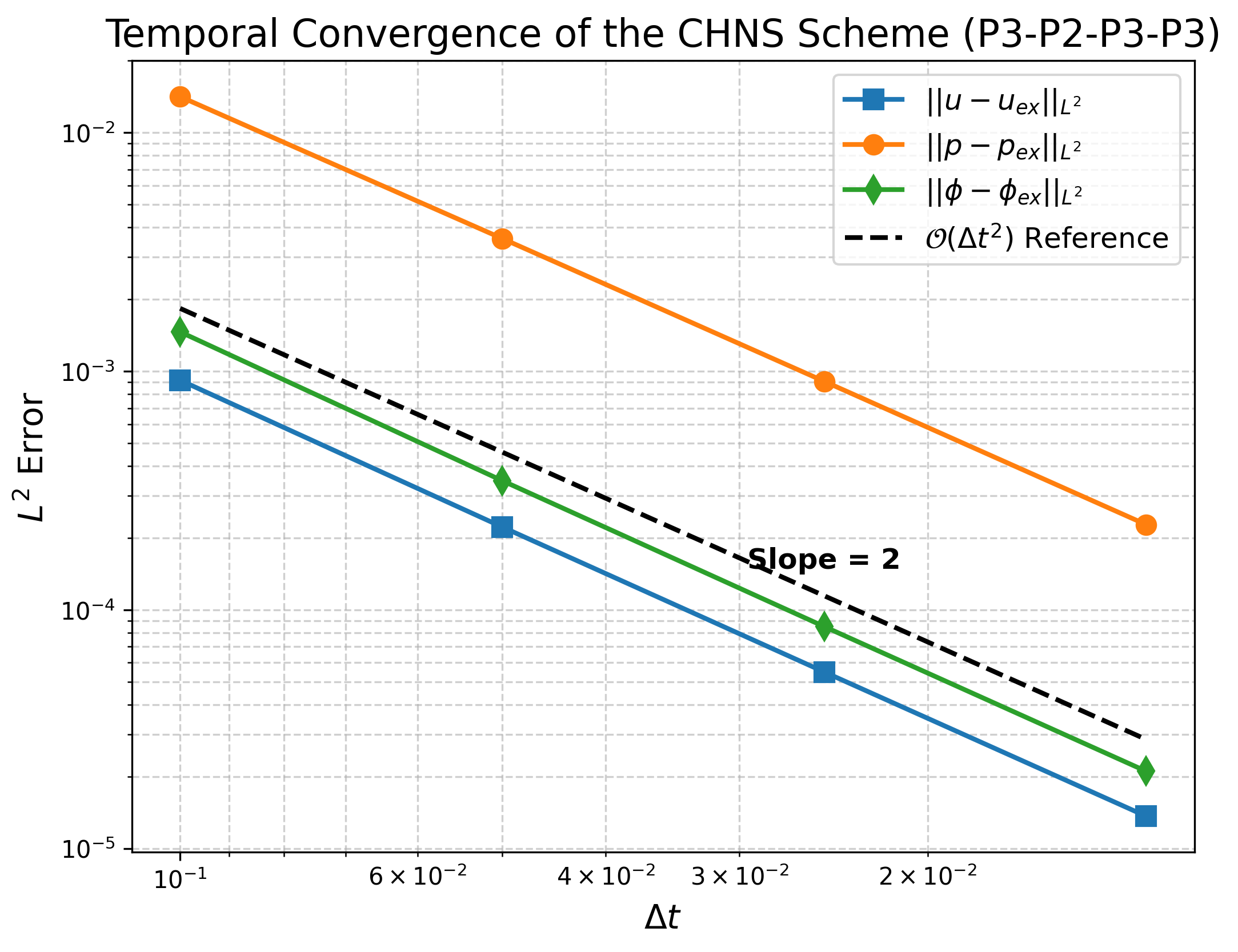}
        \caption{Temporal error decay ($P_3$-based) consistent with approximately second-order temporal accuracy over the tested range.}
        \label{fig:temp_conv_deg3}
    \end{minipage}
    
    \label{fig:temp_conv}
\end{figure}

\subsection{Robustness of the Scheme}

\subsubsection{Temporal Cauchy Convergence Test}

To further examine the temporal behavior of the scheme in a nonlinear flow configuration, we perform a temporal Cauchy convergence study for the two-phase lid-driven cavity problem. The computational domain $\Omega=[0,1]^2$ is discretized using a fixed $128\times128$ uniform mesh. We set the viscosity to $\nu=0.01$, the maximum lid velocity to $U=1.0$, and compute solutions with the sequence of halved time steps
\[
\Delta t\in
\left\{
2.5\times10^{-3},
1.25\times10^{-3},
6.25\times10^{-4},
3.125\times10^{-4}
\right\}.
\]

Because the three-level method requires two starting solution levels, a first-order semi-implicit backward-Euler step is used only to generate the first numerical solution. Although backward Euler is first-order accurate as a time-stepping method, a single starting step produces an $\mathcal{O}(\Delta t^2)$ one-step solution error under sufficient regularity, which is compatible with initialization of a globally second-order multistep method.

For the principal parameter set
\[
\theta\in\{0.51,0.75,1.0\},
\qquad
\epsilon\in\{0,0.5\nu,\nu,2\nu\},
\]
the measured interior Cauchy rates are close to second order for most configurations and for the velocity, phase, and pressure variables. The detailed tables are omitted here for brevity. Related second-order CHNS time discretizations can be found in \cite{HanWang2017, LiShen2022_MSAV}; the purpose of the present Cauchy study is specifically to verify that the implementation of the proposed scheme retains its expected temporal behavior in a strongly coupled nonlinear benchmark.

The observed rates indicate that, for moderate values of the regularization parameter, the curvature regularization and extrapolated treatment of the nonlinear terms do not introduce an evident first-order temporal contamination over the tested range.

Additional exploratory computations were performed with substantially larger regularization parameters, including $\epsilon=5\nu$ and $\epsilon=10\nu$. In these over-regularized cases, some apparent Cauchy rates become highly irregular and can exceed $3$ or even $10$ during the finest refinements. We do not interpret these values as genuine high-order convergence. A plausible explanation is that strong regularization suppresses temporal variations to the point that successive Cauchy differences become extremely small. The rate
\[
r_{\Delta t}
=
\log_2
\left(
\frac{E_{\Delta t}}{E_{\Delta t/2}}
\right)
\]
then becomes sensitive to cancellation, round-off effects, and pre-asymptotic behavior whenever the denominator is anomalously small. The irregular rates observed for very large $\epsilon$ should therefore be viewed as evidence that these runs have left the useful asymptotic regime rather than as an increase in the formal order of the method.

\subsubsection{Dependence of Performance on $\epsilon$ and $\theta$}

To investigate the practical effect of the temporal-curvature regularization parameter $\epsilon$ and the time-stepping parameter $\theta$, we examine the energy and enstrophy dynamics of the two-phase lid-driven cavity problem over a range of numerical configurations. Computations were performed for
\[
\nu\in\left\{\frac{1}{250},\frac{1}{500},\frac{1}{1000}\right\}.
\]
The principal qualitative behavior was similar across these cases, and we therefore present the representative case $\nu=1/500$.

For the energy diagnostic, we compute the physical CHNS energy directly from the numerical phase and velocity fields,
\[
E_{\mathrm{phys}}(t)
=
\int_\Omega
\left[
\frac12|\mathbf{u}|^2
+
\frac{\lambda}{2}|\nabla\phi|^2
+
\frac{\lambda}{4\eta^2}(\phi^2-1)^2
\right]dx.
\]
The enstrophy diagnostic is computed from the squared $L^2$ magnitude of the two-dimensional vorticity field.

\begin{figure}[htpb]
    \centering
    \includegraphics[width=\textwidth]{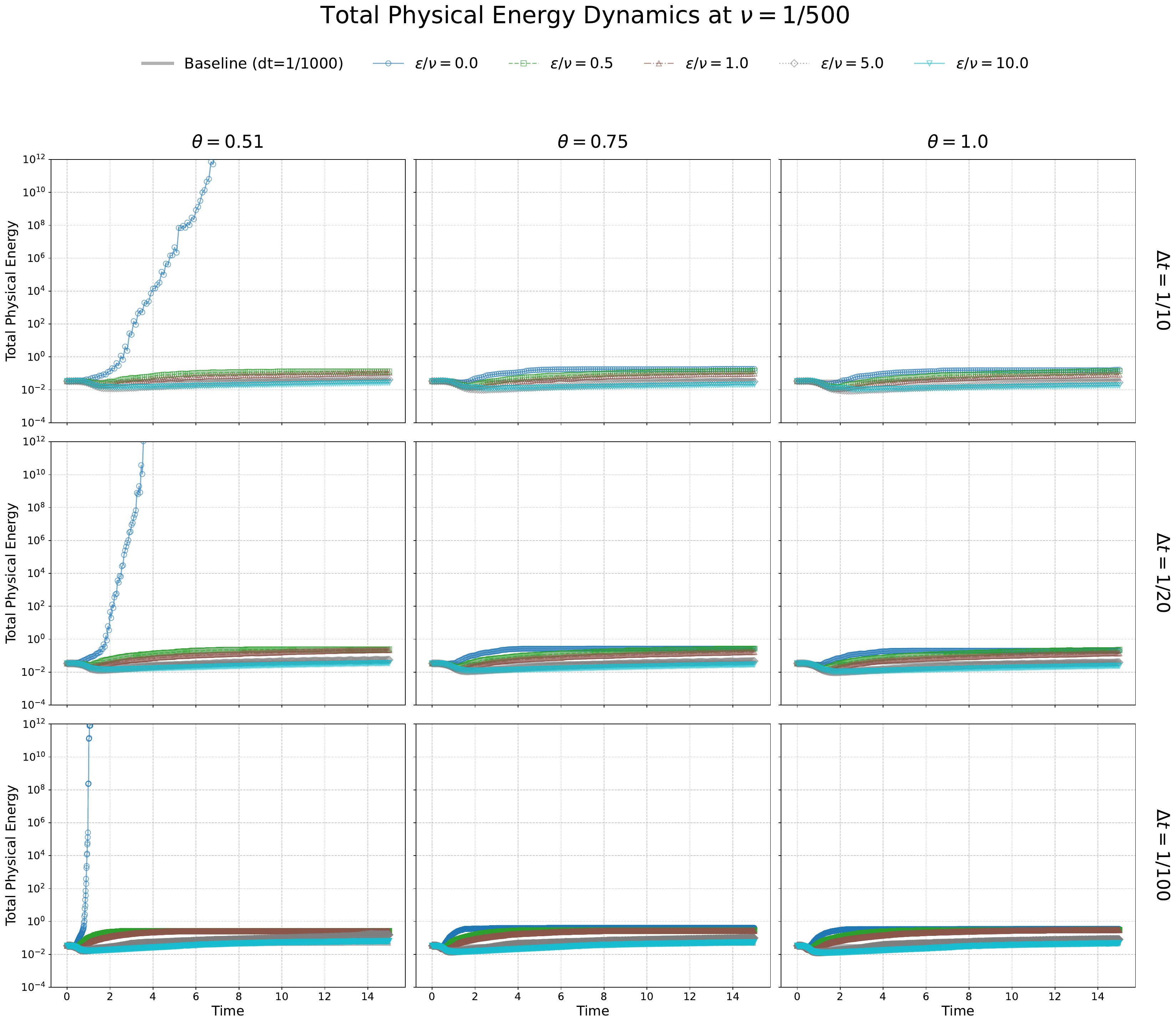}
    \caption{Evolution of the total physical energy for the two-phase lid-driven cavity flow at $\nu=1/500$ for different values of $\theta$, $\epsilon$, and $\Delta t$.}
    \label{fig:robustness_energy_nu500}
\end{figure}

\begin{figure}[htpb]
    \centering
    \includegraphics[width=\textwidth]{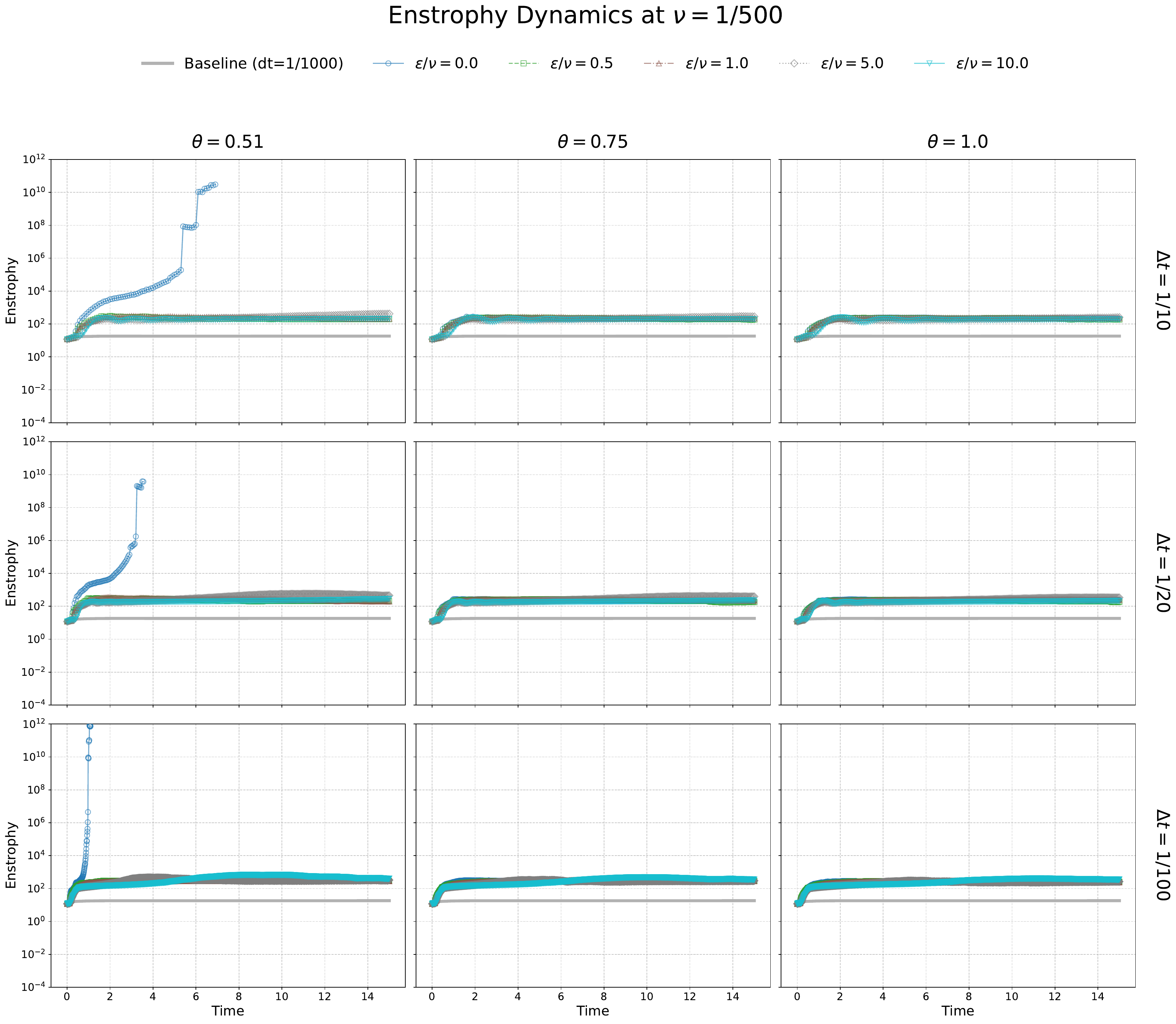}
    \caption{Evolution of the enstrophy for the two-phase lid-driven cavity flow at $\nu=1/500$. The regularized calculations limit the rapid enstrophy growth observed in the unstable unregularized runs.}
    \label{fig:robustness_enstrophy_nu500}
\end{figure}

Figures \ref{fig:robustness_energy_nu500} and \ref{fig:robustness_enstrophy_nu500} show the evolution of physical energy and enstrophy for
\[
\theta\in\{0.51,0.75,1.0\},
\qquad
\Delta t\in\left\{\frac1{10},\frac1{20},\frac1{100}\right\},
\]
with different values of $\epsilon$. These experiments are motivated by the curvature-regularization framework in \cite{JiangMohebujjaman2016, JiangYang2024}.

The unconditional energy-stability result established in Theorem \ref{main-th} is derived under homogeneous velocity boundary conditions,
\[
\mathbf{u}|_{\partial\Omega}=0.
\]
The lid-driven cavity problem does not satisfy this assumption because the moving upper boundary continuously performs work on the fluid. Consequently, the physical energy of this driven problem is not expected to decrease monotonically. Instead, a stable computation should remain bounded and approach a dynamically balanced regime in which boundary energy input is compensated by viscous and phase-field dissipation.

The most sensitive configuration occurs for $\theta=0.51$ and $\epsilon=0$, corresponding to a weakly dissipative time discretization close to the Crank--Nicolson limit. For this case, the computations eventually lose boundedness for all three displayed time-step sizes. A notable and initially counterintuitive observation is that the instability occurs earlier in physical time as $\Delta t$ is reduced. Thus, in this parameter regime, simply refining the time step does not restore robustness.

One possible interpretation of this behavior is associated with a weakly damped parasitic temporal component of the unregularized two-step discretization near the Crank--Nicolson limit. When the intrinsic numerical damping is very small, oscillatory components generated by the strongly coupled interfacial dynamics and by the continuous energy input from the moving lid can persist over many consecutive updates. Reducing $\Delta t$ increases the number of such updates over a fixed physical interval and, in the absence of additional regularization, may allow this weakly damped component to accumulate rather than decay. This provides a plausible explanation for the observed earlier loss of boundedness at smaller $\Delta t$. We emphasize that this interpretation is based on the observed numerical behavior and should not be regarded as a rigorous instability mechanism established by the analysis in Section \ref{uncon_stab}.

The corresponding enstrophy curves exhibit rapid growth when these unregularized computations become unstable. The growth is consistent with the development of increasingly strong small-scale velocity gradients as the numerical solution loses robustness.

Introducing positive curvature regularization substantially changes this behavior. For moderate values such as
\[
\epsilon=0.5\nu
\qquad\text{and}\qquad
\epsilon=\nu,
\]
the energy and enstrophy remain bounded over the tested interval, including for the relatively large time steps considered here. The results suggest that the regularization effectively damps the temporal oscillatory components responsible for the loss of robustness in the unregularized $\theta=0.51$ calculations.

Increasing $\theta$ also increases the dissipative character of the time discretization. In particular, $\theta=1$ corresponds to the BDF2 limit of the underlying three-level derivative and produces substantially more robust behavior even when $\epsilon=0$. However, increasing $\theta$ also changes the numerical dissipation of the base time integrator. The two-parameter formulation therefore provides additional flexibility: $\theta$ controls the character of the underlying second-order method, while $\epsilon$ supplies a separate curvature-regularization mechanism. The numerical results indicate that moderate positive values of $\epsilon$ can improve robustness when $\theta$ is chosen close to the weakly dissipative Crank--Nicolson limit.

\subsection{Spinodal Decomposition, Energy Dissipation, and Mass Conservation}

To examine long-time behavior, energy dissipation, and phase-mass conservation, we simulate spinodal decomposition of a binary fluid mixture. This classical benchmark describes the spontaneous separation of a nearly homogeneous, thermodynamically unstable mixture into distinct phase domains.

The computation is performed on the unit square,
\[
\Omega=[0,1]^2,
\]
using a $256\times256$ uniform triangular mesh. The initial phase field is prescribed as
\[
\phi_0(\mathbf{x})
=
0.2-0.01\,r(\mathbf{x}),
\]
where $r(\mathbf{x})$ is a realization of uniformly distributed random noise on $[-1,1]$. The initial velocity and pressure are set to zero, and no-slip boundary conditions are imposed on all walls. The parameters are
\[
\nu=0.1,\qquad
\eta=0.02,\qquad
\lambda=0.001,\qquad
M=0.01,\qquad
\epsilon=10^{-5}.
\]

\begin{figure}[h!]
    \centering
    \includegraphics[width=0.95\textwidth]{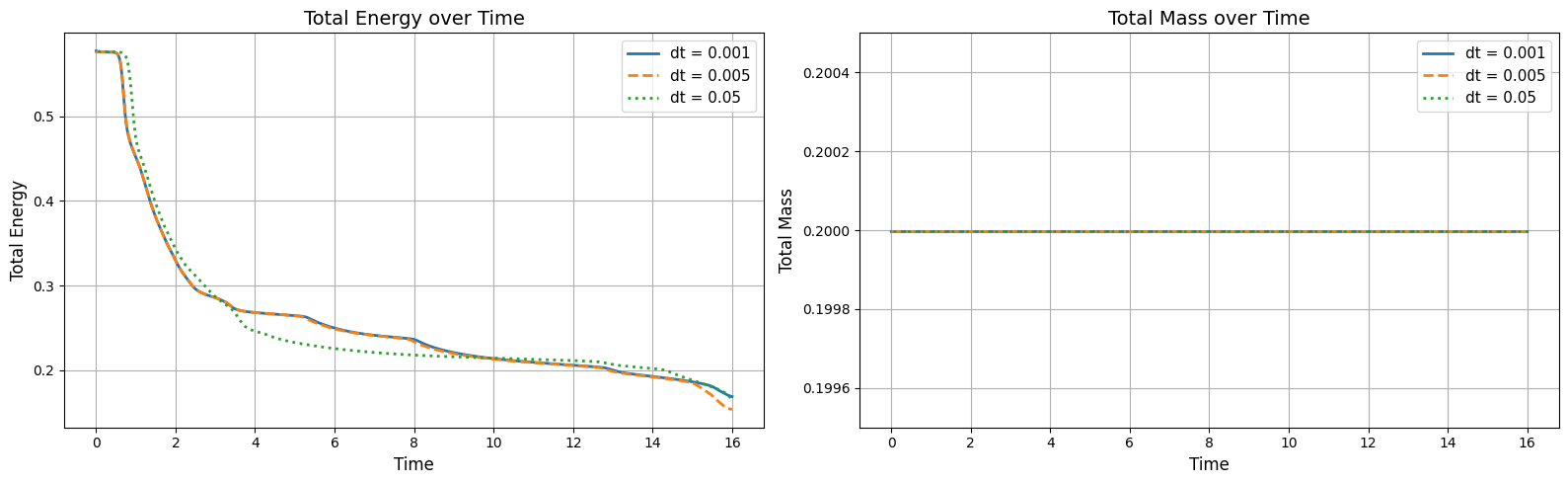}
    \caption{Comparison of the computed energy (left) and total phase mass (right) for $\Delta t=0.05$, $0.005$, and $0.001$. The mass remains constant at its numerically realized initial value, while the energy curves for $\Delta t=0.005$ and $\Delta t=0.001$ are nearly indistinguishable on the plotted scale.}
    \label{fig:spinodal_dt_comparison}
\end{figure}

To select a time step for the long-time simulation, we first compare the evolution of the computed energy and total phase mass for
\[
\Delta t=0.05,\qquad
0.005,\qquad
0.001.
\]
As shown in Figure \ref{fig:spinodal_dt_comparison}, the total phase mass remains constant, to numerical precision, at its initial value, which is approximately $0.20$ for the random realization used here. The computed energy decreases monotonically for each of the three time steps, consistent with the expected dissipative behavior of the scheme.

The energy curves for $\Delta t=0.005$ and $\Delta t=0.001$ are visually almost indistinguishable on the scale of the figure. This indicates that the energy evolution is effectively time-step converged at the plotted resolution. We therefore use
\[
\Delta t=0.005
\]
for the long-time calculation up to
\[
T=100.
\]

Figures \ref{fig:spinodal_energy_mass_part1} and \ref{fig:spinodal_energy_mass_part2} show the evolution of the phase field together with the corresponding computed energy and phase mass. During the early stage, approximately $0\leq t\leq20$, the initially mixed state rapidly separates and develops an interconnected network of phase domains. At later times, the system enters a coarsening regime in which highly curved structures relax, thin connections break, and smaller domains merge into larger ones. This evolution is accompanied by continued energy dissipation while the phase mass remains constant to numerical precision over the full interval $0\leq t\leq100$.

\begin{figure}[htbp]
    \centering

    \begin{minipage}{0.31\textwidth}
        \centering
        \includegraphics[width=\linewidth]{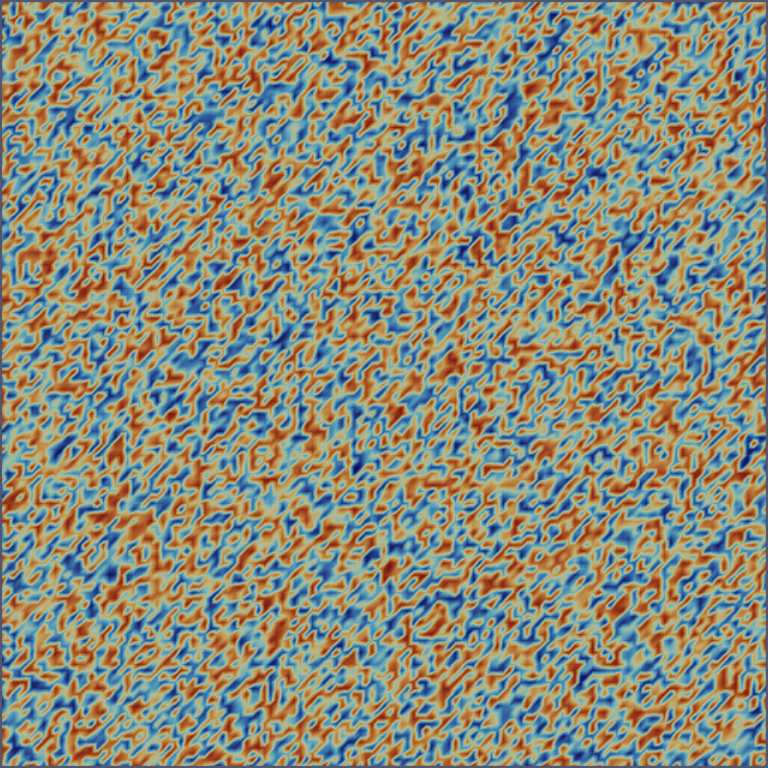}
        \centerline{(a) Phase field at $t=0$}
    \end{minipage}\hfill
    \begin{minipage}{0.60\textwidth}
        \centering
        \includegraphics[width=\linewidth]{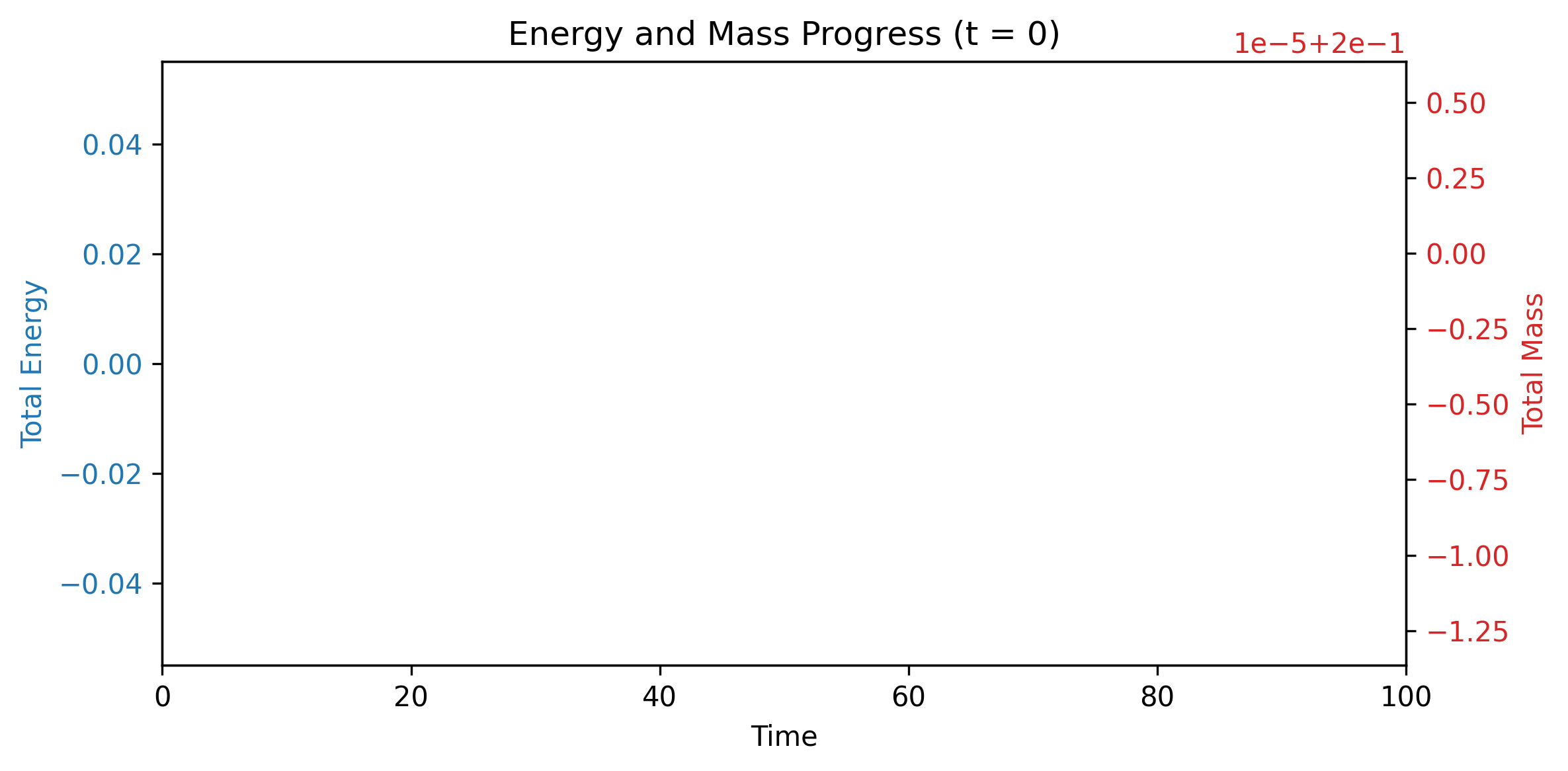}
        \centerline{(b) Energy and mass up to $t=0$}
    \end{minipage}

    \vspace{0.3cm}

    \begin{minipage}{0.31\textwidth}
        \centering
        \includegraphics[width=\linewidth]{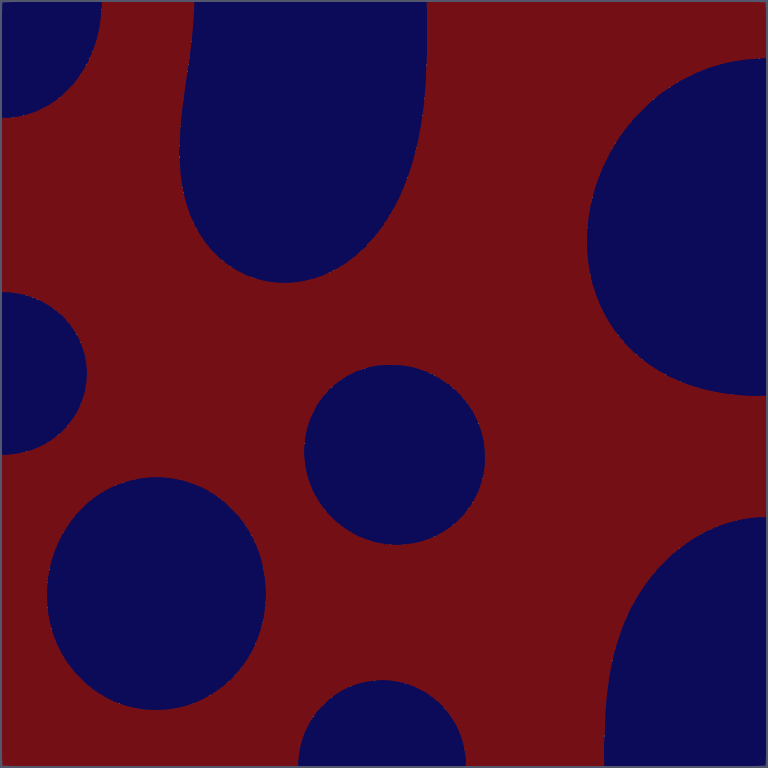}
        \centerline{(c) Phase field at $t=10$}
    \end{minipage}\hfill
    \begin{minipage}{0.60\textwidth}
        \centering
        \includegraphics[width=\linewidth]{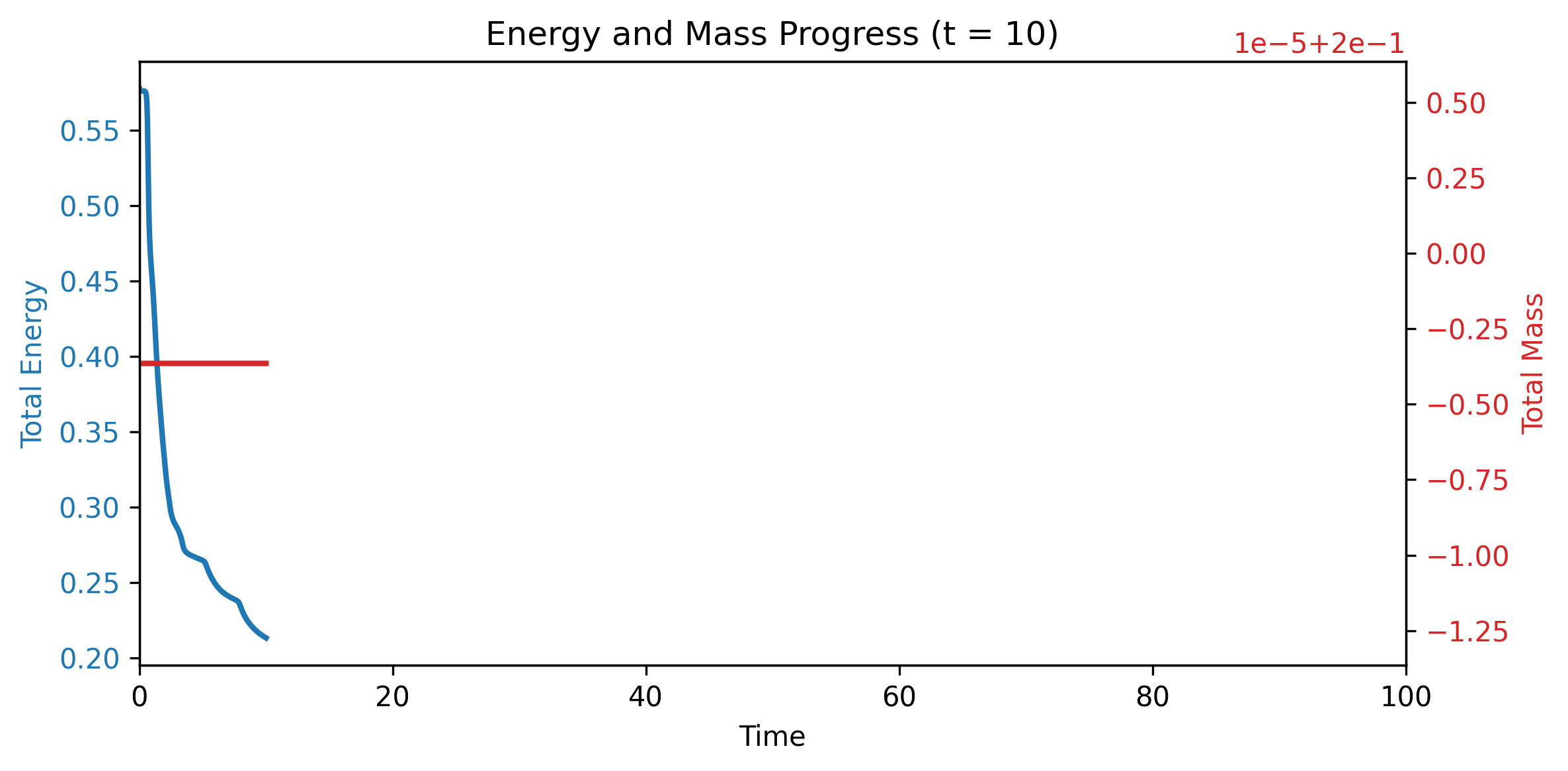}
        \centerline{(d) Energy and mass up to $t=10$}
    \end{minipage}

    \vspace{0.3cm}

    \begin{minipage}{0.31\textwidth}
        \centering
        \includegraphics[width=\linewidth]{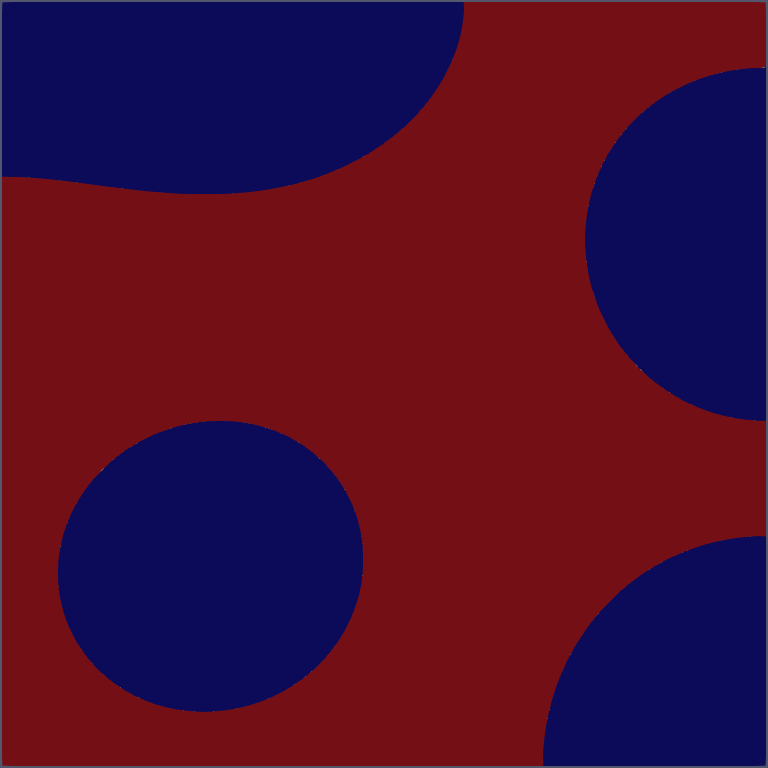}
        \centerline{(e) Phase field at $t=20$}
    \end{minipage}\hfill
    \begin{minipage}{0.60\textwidth}
        \centering
        \includegraphics[width=\linewidth]{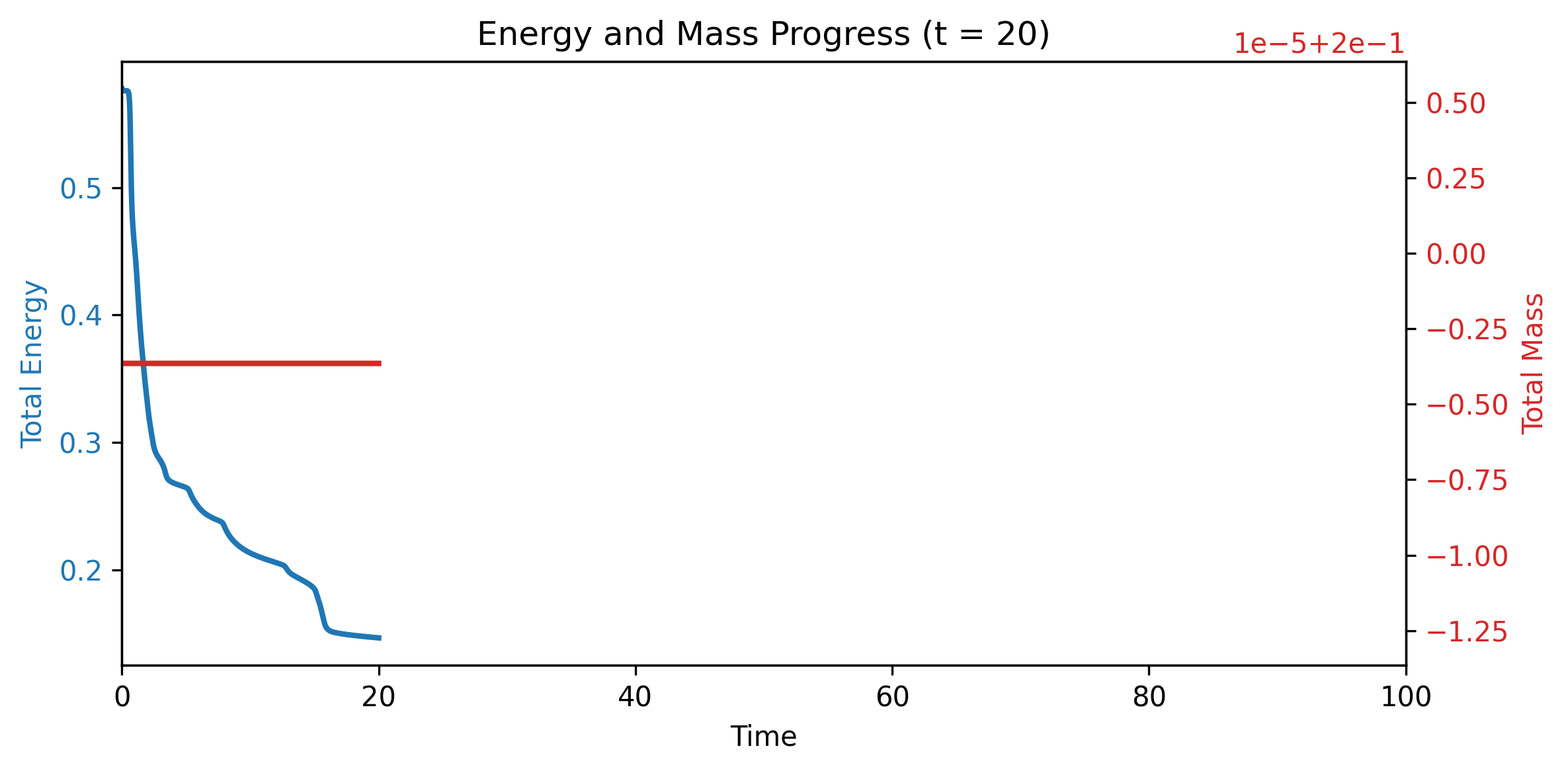}
        \centerline{(f) Energy and mass up to $t=20$}
    \end{minipage}

    \caption{Early-stage spinodal decomposition ($0\leq t\leq20$). The left column shows the evolution of the phase field, while the right column shows the corresponding computed energy and phase mass.}
    \label{fig:spinodal_energy_mass_part1}
\end{figure}

\clearpage

\begin{figure}[htpb]
    \centering

    \begin{minipage}{0.32\textwidth}
        \centering
        \includegraphics[width=\linewidth]{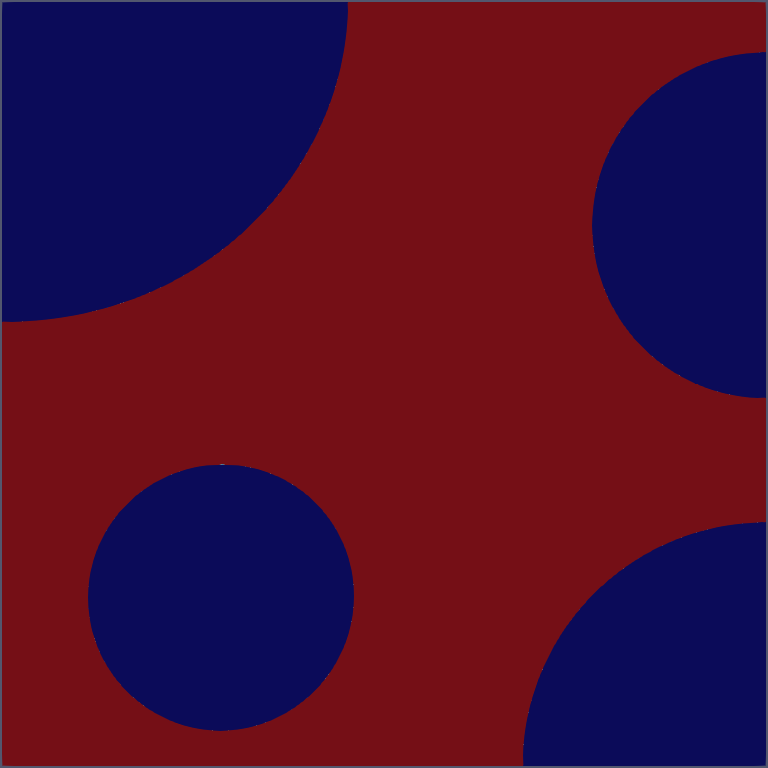}
        \centerline{(a) Phase field at $t=40$}
    \end{minipage}\hfill
    \begin{minipage}{0.60\textwidth}
        \centering
        \includegraphics[width=\linewidth]{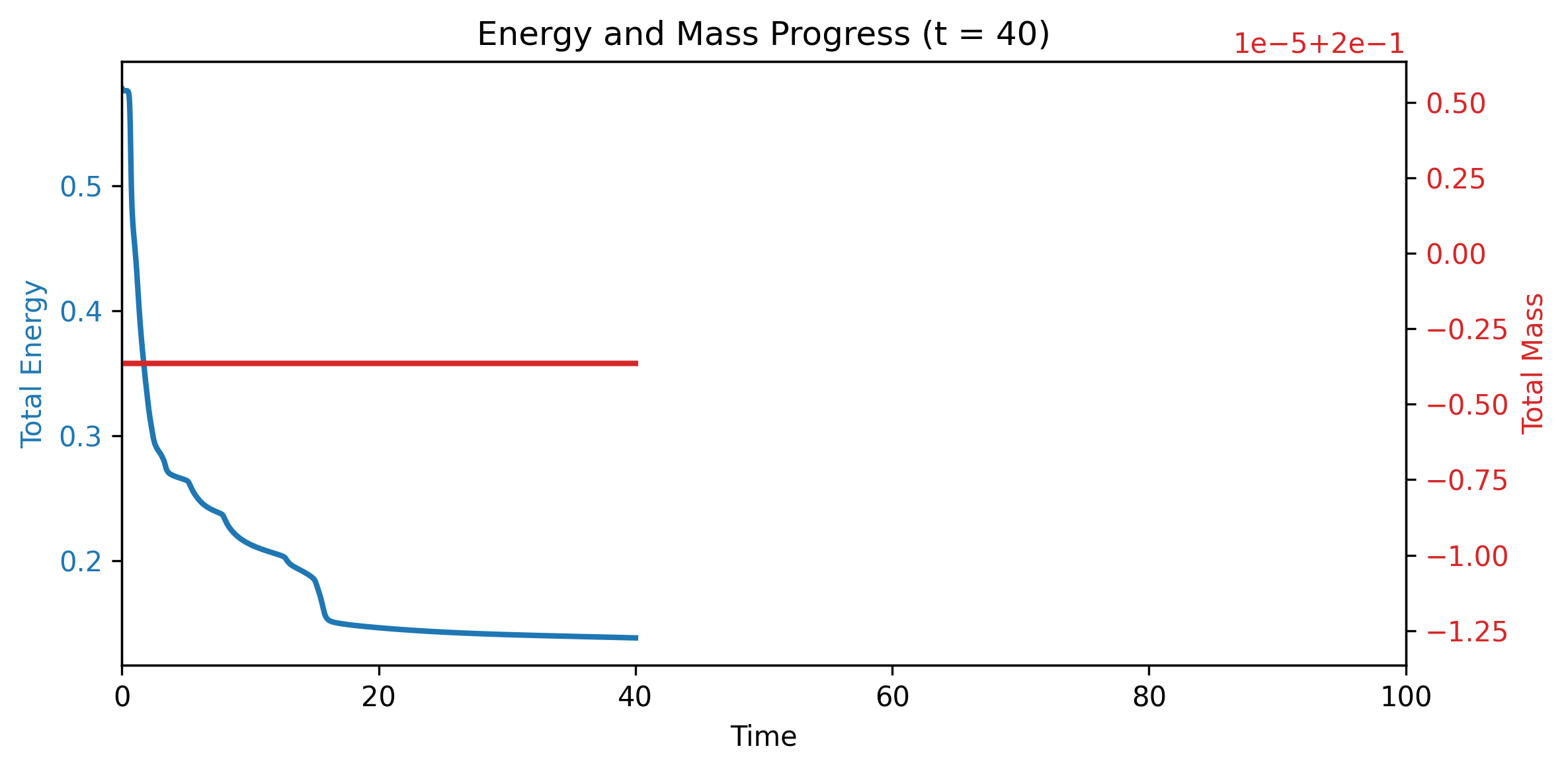}
        \centerline{(b) Energy and mass up to $t=40$}
    \end{minipage}

    \vspace{0.3cm}

    \begin{minipage}{0.32\textwidth}
        \centering
        \includegraphics[width=\linewidth]{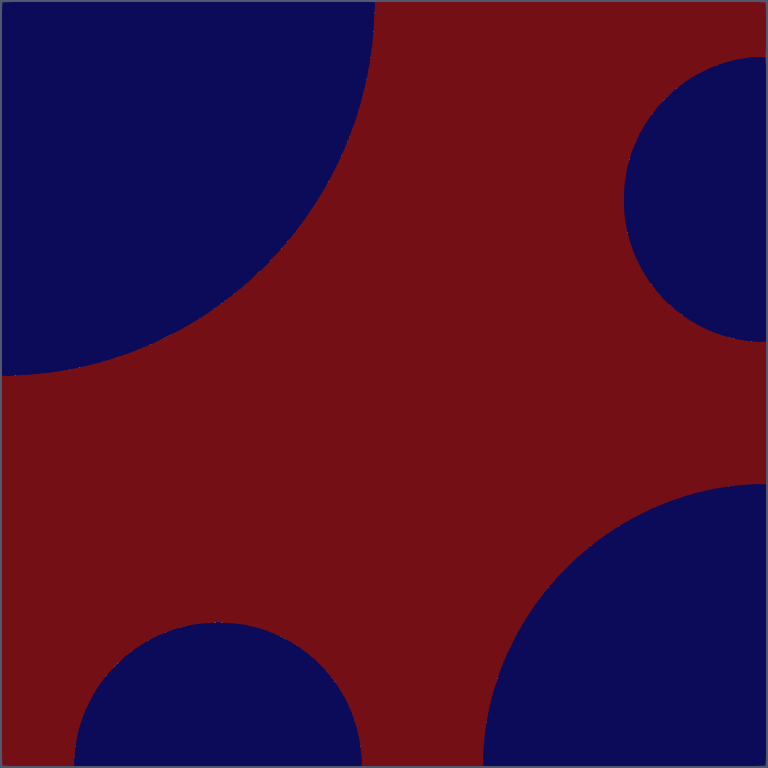}
        \centerline{(c) Phase field at $t=80$}
    \end{minipage}\hfill
    \begin{minipage}{0.60\textwidth}
        \centering
        \includegraphics[width=\linewidth]{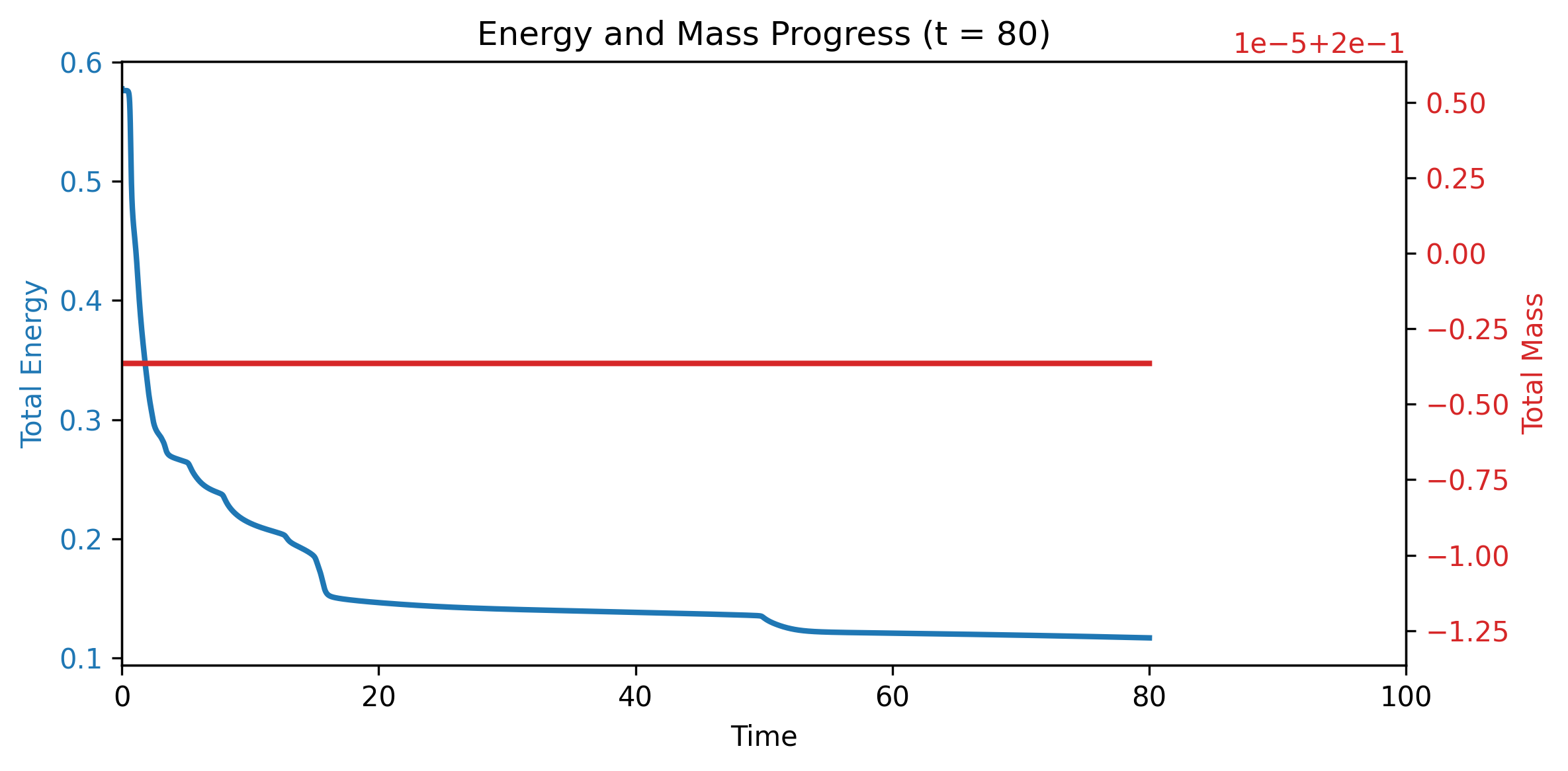}
        \centerline{(d) Energy and mass up to $t=80$}
    \end{minipage}

    \vspace{0.3cm}

    \begin{minipage}{0.32\textwidth}
        \centering
        \includegraphics[width=\linewidth]{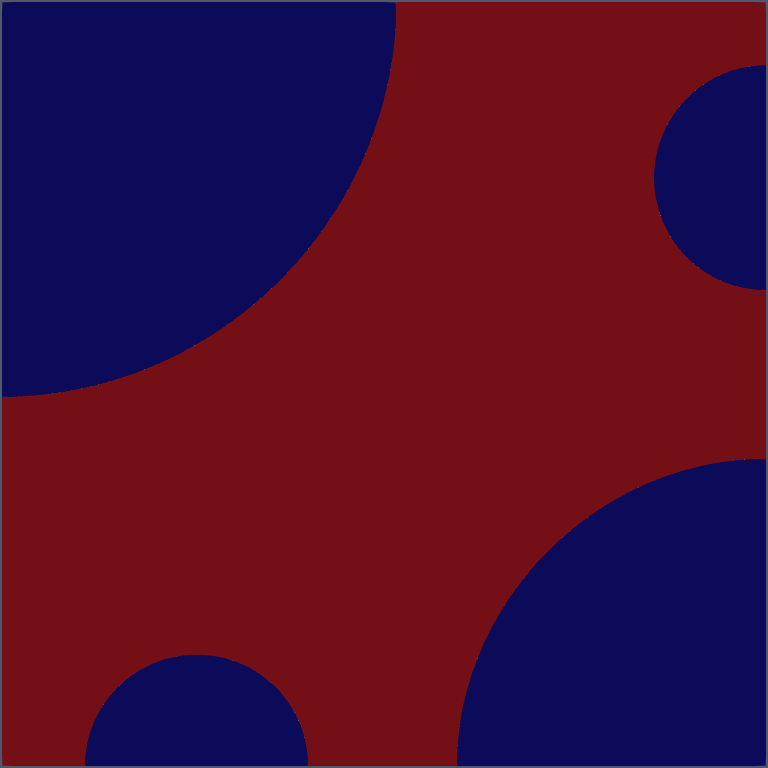}
        \centerline{(e) Phase field at $t=100$}
    \end{minipage}\hfill
    \begin{minipage}{0.60\textwidth}
        \centering
        \includegraphics[width=\linewidth]{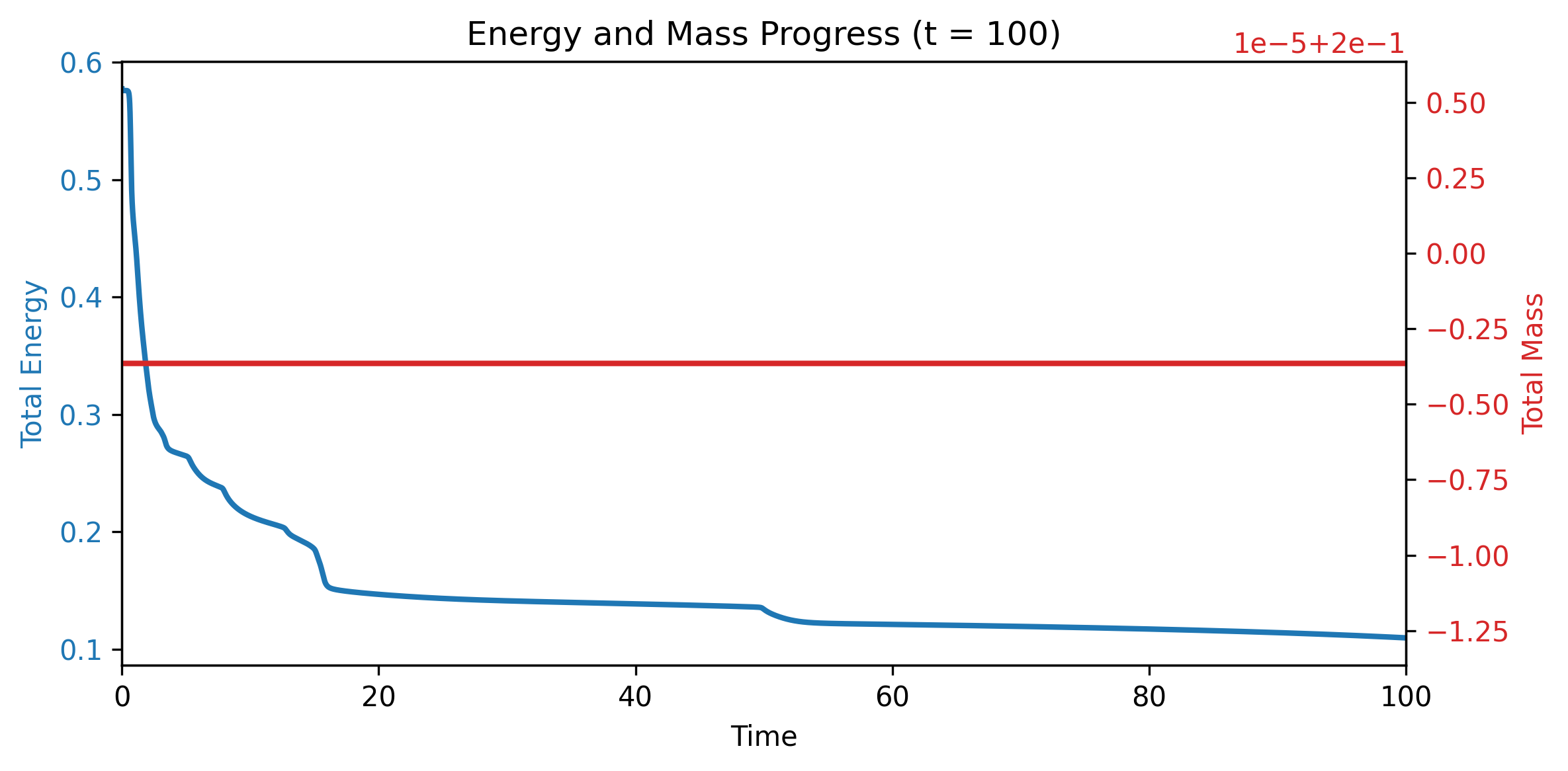}
        \centerline{(f) Energy and mass up to $t=100$}
    \end{minipage}

    \caption{Late-stage spinodal decomposition ($40\leq t\leq100$). The system undergoes continued coarsening while the computed energy decreases and the phase mass remains conserved to numerical precision.}
    \label{fig:spinodal_energy_mass_part2}
\end{figure}

\subsection{Shape Relaxation of a Square Droplet}

To evaluate the ability of the algorithm to resolve surface-tension-driven interfacial motion, we consider the relaxation of an initially square droplet. At fixed enclosed area, the circular configuration minimizes interfacial length. Consequently, the large local curvature associated with the corners of the initial square generates chemical-potential gradients that drive the interface toward a smoother configuration.

The simulation is performed on
\[
\Omega=[0,1]^2
\]
using a $256\times256$ uniform triangular mesh. The discrete phase field is initialized by assigning
\[
\phi=1
\]
inside
\[
[0.4,0.6]\times[0.4,0.6]
\]
and
\[
\phi=-1
\]
outside this region. In the finite-element representation, this piecewise initialization is resolved over the underlying mesh. The initial velocity and pressure fields are zero, and no-slip boundary conditions are imposed on all walls. The specified parameters are
\[
\Delta t=0.005,\qquad
\nu=0.1,\qquad
\eta=0.005,\qquad
\lambda=2.5\times10^{-5},\qquad
M=0.01.
\]

Figure \ref{fig:shape_relax} shows the evolution of the phase field. The corners of the initial square, Figure \ref{fig:shape_relax}(a), rapidly smooth as the interface relaxes. The droplet then passes through a sequence of rounded-square configurations for approximately $0.035\leq t\leq0.14$ and approaches an approximately circular equilibrium by $t=1.0$. The computation remains stable throughout the simulated interval and preserves the enclosed phase mass to numerical precision.

\begin{figure}[htpb]
    \centering
    \begin{minipage}{0.32\textwidth}
        \centering
        \includegraphics[width=\linewidth]{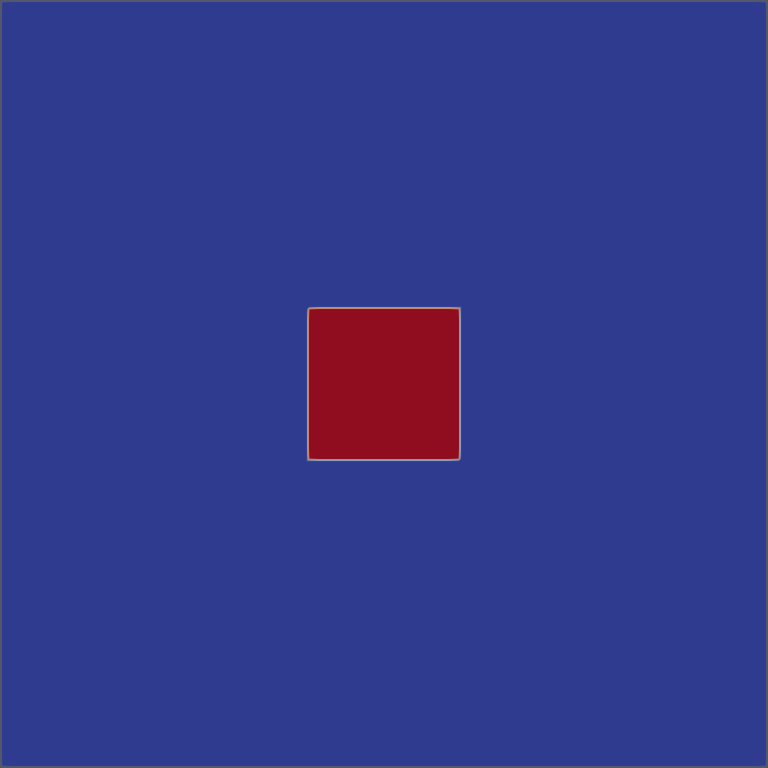}
        \centerline{(a) $t=0.00$}
    \end{minipage}\hfill
    \begin{minipage}{0.32\textwidth}
        \centering
        \includegraphics[width=\linewidth]{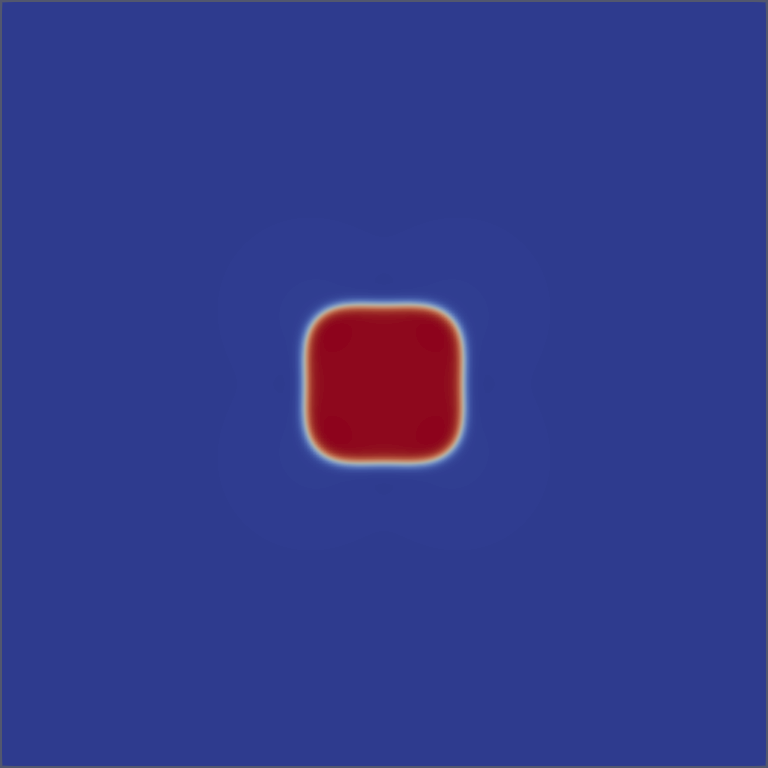}
        \centerline{(b) $t=0.035$}
    \end{minipage}\hfill
    \begin{minipage}{0.32\textwidth}
        \centering
        \includegraphics[width=\linewidth]{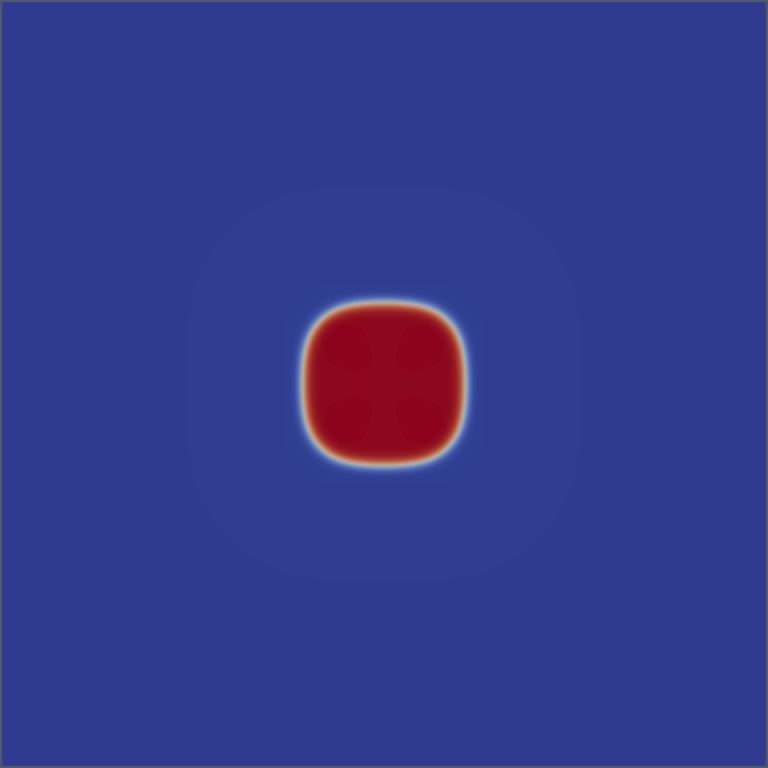}
        \centerline{(c) $t=0.07$}
    \end{minipage}

    \vspace{0.3cm}

    \begin{minipage}{0.32\textwidth}
        \centering
        \includegraphics[width=\linewidth]{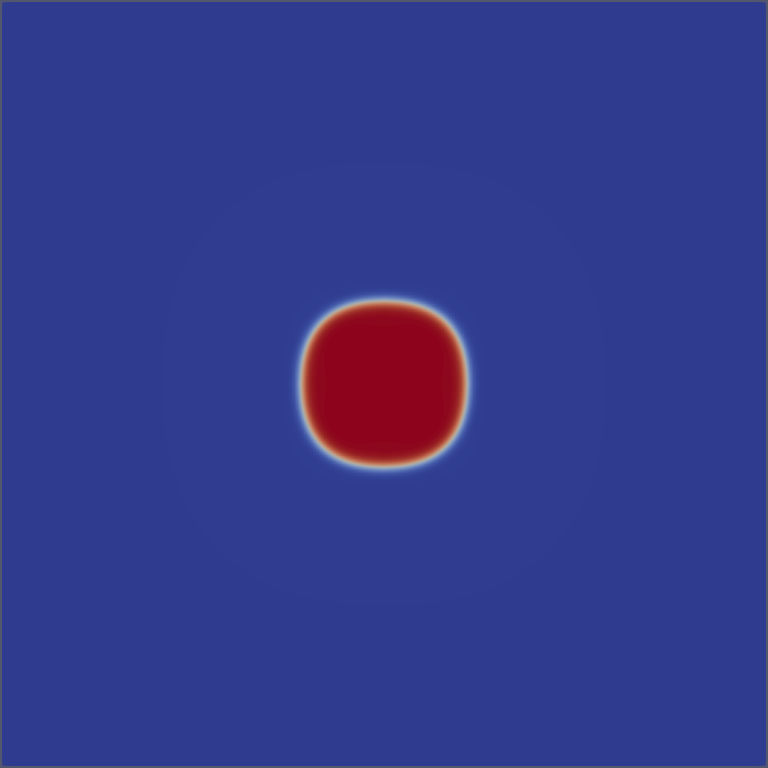}
        \centerline{(d) $t=0.105$}
    \end{minipage}\hfill
    \begin{minipage}{0.32\textwidth}
        \centering
        \includegraphics[width=\linewidth]{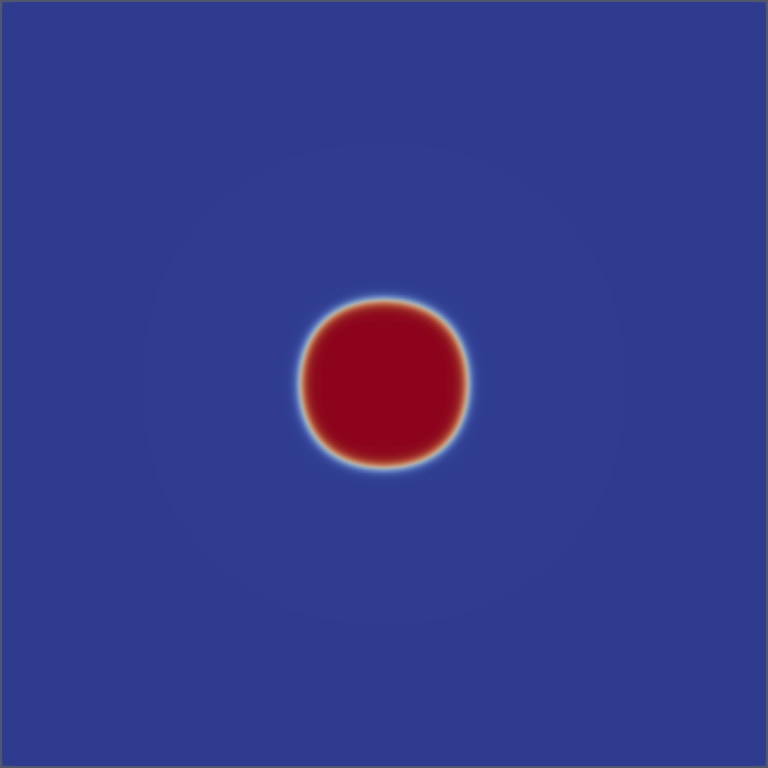}
        \centerline{(e) $t=0.14$}
    \end{minipage}\hfill
    \begin{minipage}{0.32\textwidth}
        \centering
        \includegraphics[width=\linewidth]{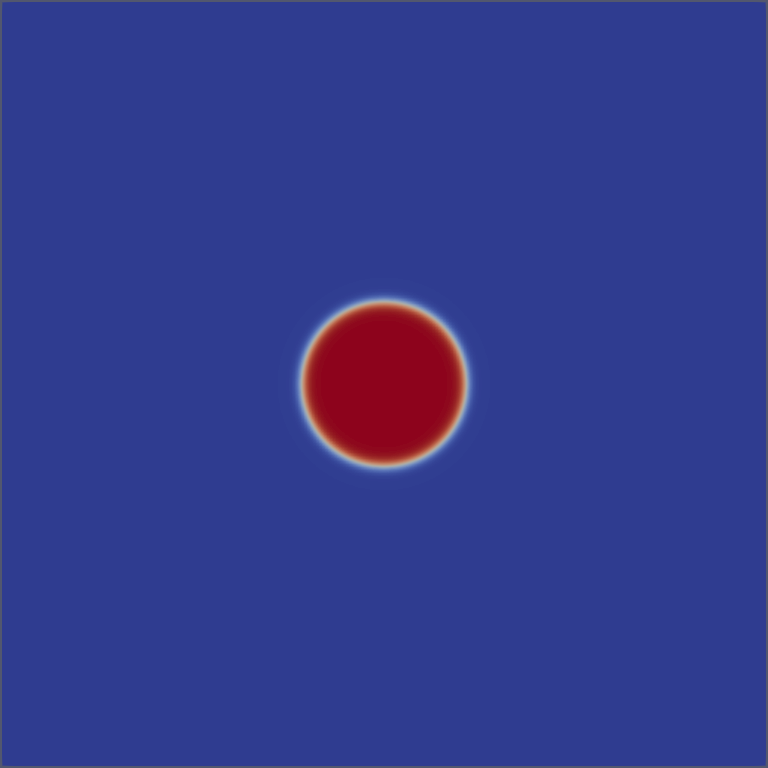}
        \centerline{(f) $t=1.0$}
    \end{minipage}

    \caption{Evolution of the phase field during relaxation of the initially square droplet. The high-curvature corners progressively smooth and the interface approaches an approximately circular equilibrium.}
    \label{fig:shape_relax}
\end{figure}

\subsection{Two-Phase Lid-Driven Cavity Flow}

To examine the behavior of the scheme under strong shear and substantial interfacial deformation, we simulate a two-phase lid-driven cavity flow on a $128\times128$ uniform mesh over
\[
\Omega=[0,1]^2.
\]
A horizontal diffuse interface is initialized at $y=0.5$ using
\[
\phi_0(x,y)
=
\tanh
\left(
\frac{0.5-y}{\sqrt{2}\eta}
\right),
\]
so that the $\phi\approx1$ phase lies below the $\phi\approx-1$ phase. No-slip conditions are imposed on the left, right, and bottom walls, while the top boundary is driven by the regularized horizontal velocity profile
\[
u(x,1)=16x^2(x-1)^2,
\]
which vanishes smoothly at the upper corners.

The parameters are
\[
\nu=0.002,\qquad
\eta=0.01,\qquad
\lambda=2\times10^{-6},\qquad
M=0.005,\qquad
\epsilon=10^{-5},
\]
and the system is integrated to $T=15$ using
\[
\Delta t=0.001
\]
with $P_2-P_1-P_2-P_2$ elements.

Figure \ref{fig:lid_driven_cavity} shows the phase-field evolution. The moving lid generates a clockwise primary vortex that deforms the initially horizontal interface. By $t=5$, the lower phase has been drawn into a pronounced interfacial wave. As the circulation develops further, approximately over $7.5\leq t\leq10$, the interface is stretched, rolled, and transported upward. At later times, the diffuse interface undergoes substantial folding and forms a thin spiraling structure. The computation remains stable over the full simulated interval and resolves the strongly deformed diffuse interface without visible spurious oscillations at the plotted resolution.

\begin{figure}[htpb]
    \centering

    \begin{minipage}{0.25\textwidth}
        \centering
        \includegraphics[width=\linewidth]{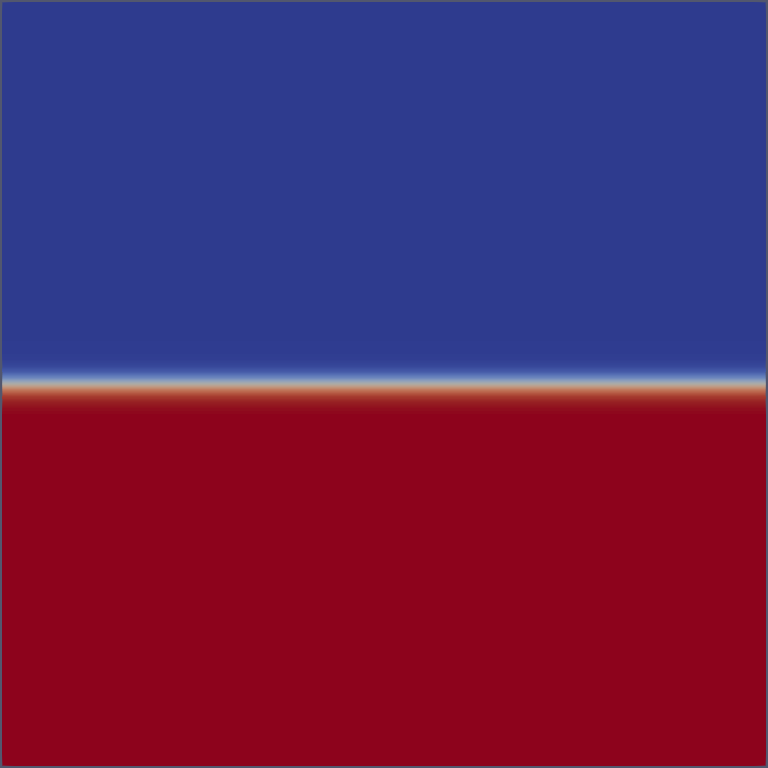}
        \centerline{(a) $t=0.0$}
    \end{minipage}\hfill
    \begin{minipage}{0.25\textwidth}
        \centering
        \includegraphics[width=\linewidth]{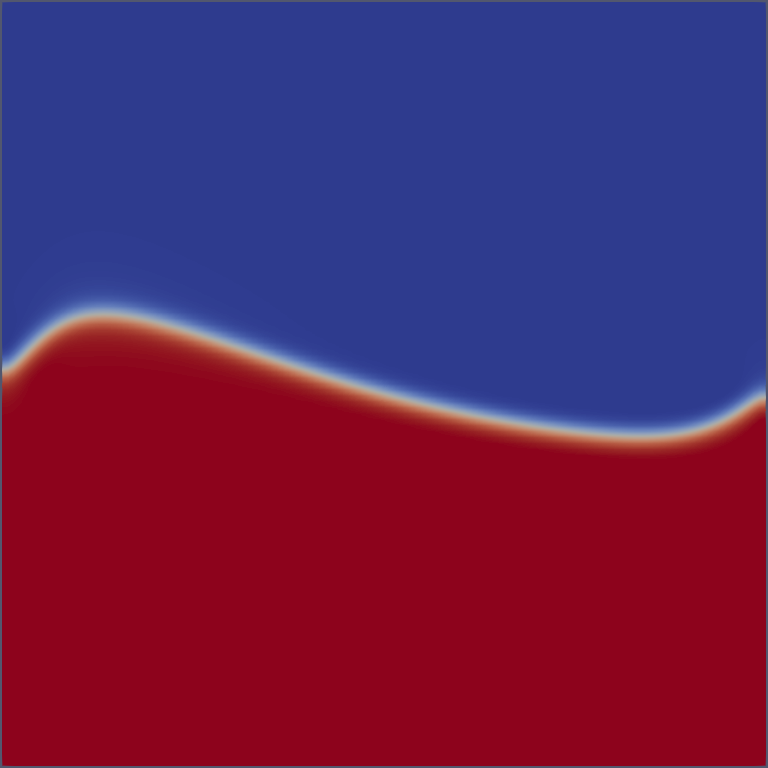}
        \centerline{(b) $t=2.5$}
    \end{minipage}\hfill
    \begin{minipage}{0.25\textwidth}
        \centering
        \includegraphics[width=\linewidth]{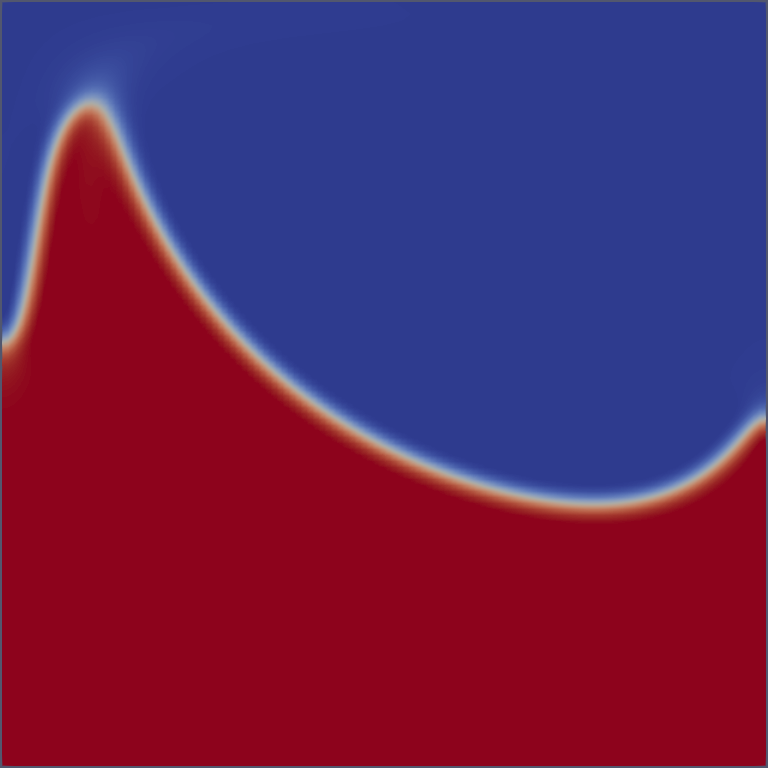}
        \centerline{(c) $t=5.0$}
    \end{minipage}

    \vspace{0.1cm}

    \begin{minipage}{0.25\textwidth}
        \centering
        \includegraphics[width=\linewidth]{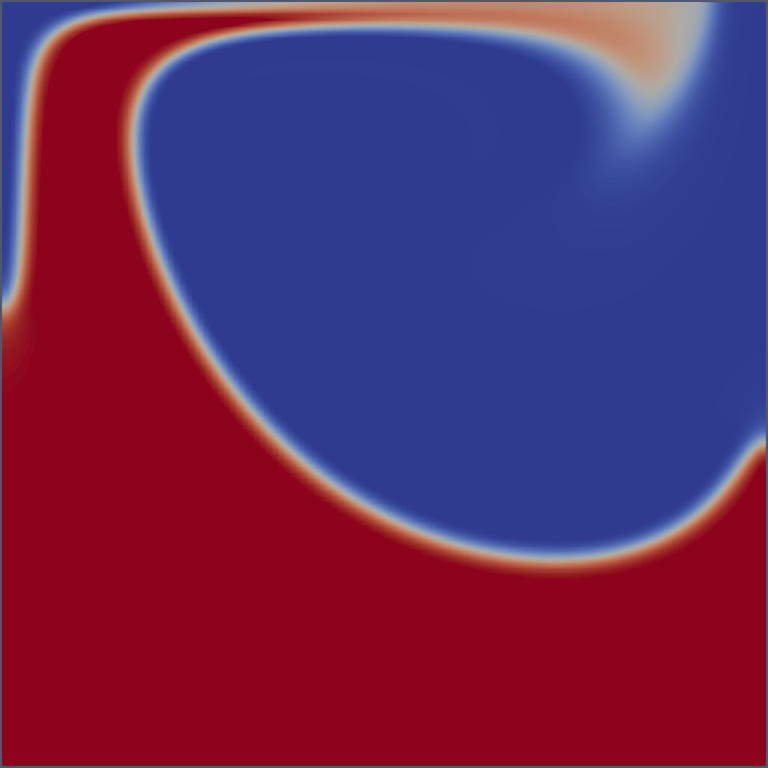}
        \centerline{(d) $t=7.5$}
    \end{minipage}\hfill
    \begin{minipage}{0.25\textwidth}
        \centering
        \includegraphics[width=\linewidth]{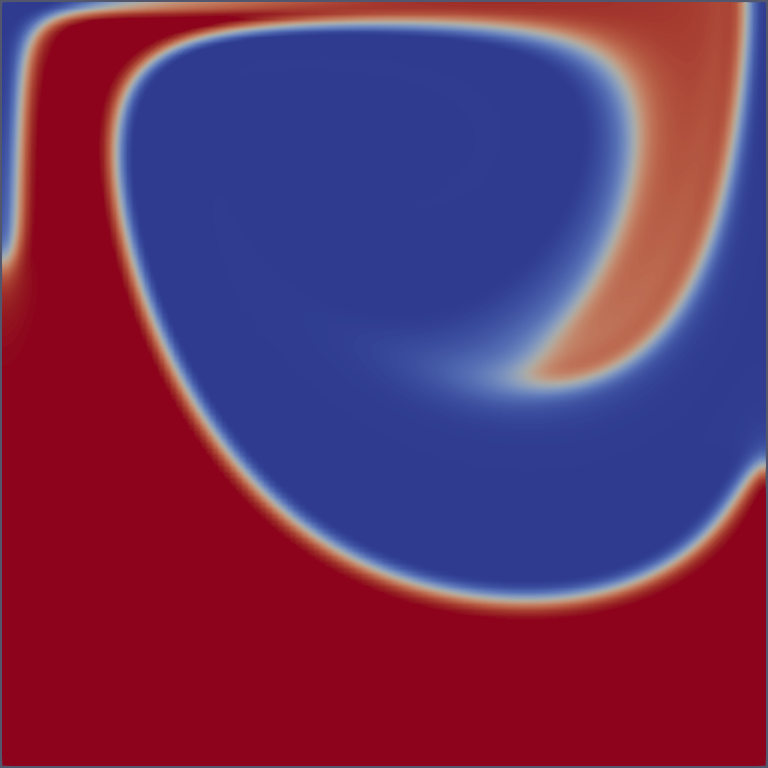}
        \centerline{(e) $t=10.0$}
    \end{minipage}\hfill
    \begin{minipage}{0.25\textwidth}
        \centering
        \includegraphics[width=\linewidth]{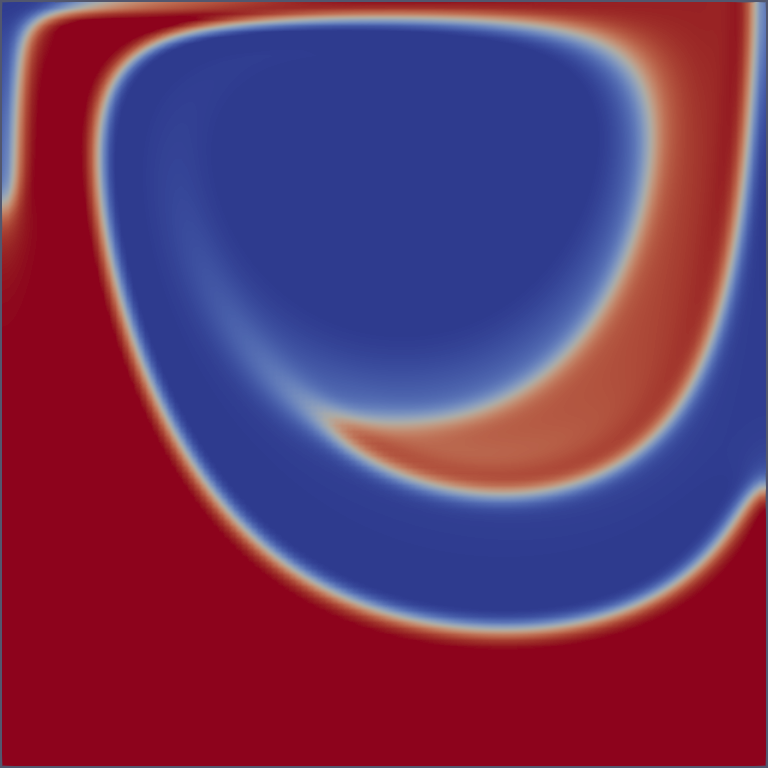}
        \centerline{(f) $t=12.5$}
    \end{minipage}

    \vspace{0.1cm}

    \begin{minipage}{0.25\textwidth}
        \centering
        \includegraphics[width=\linewidth]{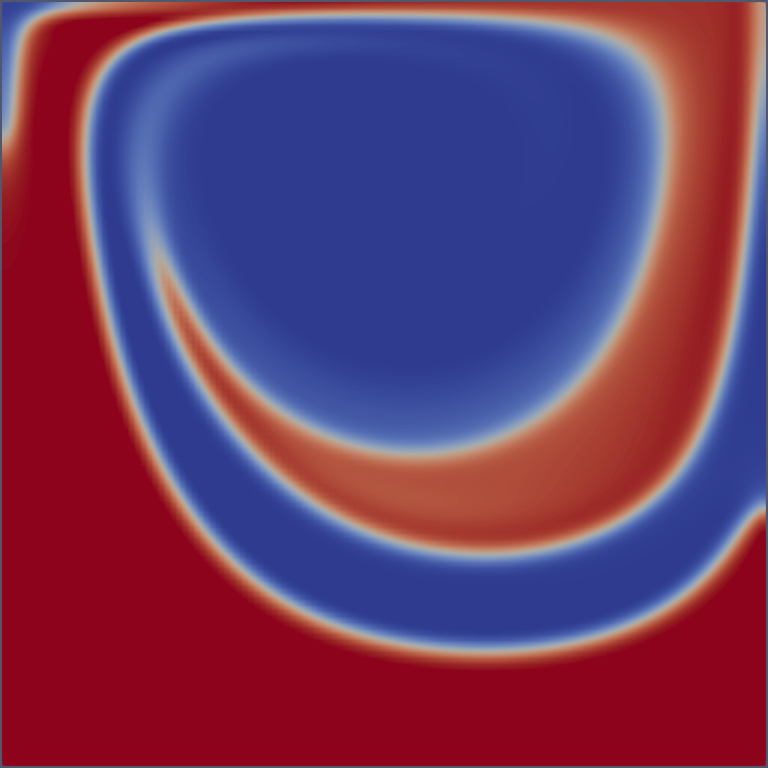}
        \centerline{(g) $t=15.0$}
    \end{minipage}

    \caption{Phase-field evolution for the two-phase lid-driven cavity problem. The clockwise circulation induced by the moving lid progressively stretches and rolls the initially horizontal interface into a strongly deformed spiraling structure.}
    \label{fig:lid_driven_cavity}
\end{figure}

\subsection{Rayleigh--Taylor Instability}

To assess the ability of the numerical method to capture strongly nonlinear interfacial dynamics driven by density differences, we consider the classical Rayleigh--Taylor instability. The instability occurs when a denser fluid is positioned above a lighter fluid in a gravitational field. For this benchmark, we use the following buoyancy-coupled phase-field system:
\begin{subequations}
    \begin{equation}
        \phi_t+\nabla\cdot(\phi\mathbf{u})
        =
        M\Delta\mu,
    \end{equation}
    \begin{equation}
        \mu
        =
        \lambda(-\Delta\phi+f(\phi)),
    \end{equation}
    \begin{equation}
        \rho_0
        \left(
        \mathbf{u}_t+\mathbf{u}\cdot\nabla\mathbf{u}
        \right)
        +
        \nabla p
        -
        \nu\Delta\mathbf{u}
        =
        -\rho(\phi)g\mathbf{j},
    \end{equation}
    \begin{equation}
        \nabla\cdot\mathbf{u}=0.
    \end{equation}
\end{subequations}
Here,
\[
f(\phi)
=
\frac{1}{\eta^2}\phi(\phi^2-1),
\]
and the density is interpolated through the phase field according to
\[
\rho(\phi)
=
\rho_1\left(\frac{1+\phi}{2}\right)
+
\rho_2\left(\frac{1-\phi}{2}\right).
\]
We denote the upward unit vector by
\[
\mathbf{j}=(0,1)^T,
\]
so that the gravitational force $-\rho(\phi)g\mathbf{j}$ acts downward.

The computation is performed on the rectangular domain
\[
\Omega=[0,1]\times[0,4]
\]
using a $128\times512$ uniform triangular mesh. The initial interface is centered at $y=2$ and perturbed according to
\[
y_0(x)
=
2+0.1\cos(2\pi x).
\]
The phase field is initialized so that its diffuse interface is centered on this perturbed curve. The denser fluid,
\[
\rho_1=3,
\]
is placed above the interface, while the lighter fluid,
\[
\rho_2=1,
\]
is placed below it. The background density is taken as
\[
\rho_0
=
\frac{\rho_1+\rho_2}{2}.
\]
The gravitational acceleration is
\[
g=10,
\]
and no-slip boundary conditions are imposed on all walls. The remaining parameters are
\[
\eta=0.01,\qquad
\lambda=10^{-4},\qquad
M=0.1.
\]
The simulation is advanced to
\[
T=1.2
\]
using
\[
\Delta t=2\times10^{-4}
\]
and $P_2-P_1-P_2-P_2$ elements. Two viscous regimes are considered:
\[
\nu=0.01
\qquad\text{and}\qquad
\nu=0.001.
\]

Figure \ref{fig:rt_instability} shows the evolution of the phase field for the two viscous regimes. For the larger viscosity, $\nu=0.01$ (top row), viscous damping suppresses small-scale interfacial roll-up. The denser fluid descends in a relatively smooth and symmetric spike while the lighter fluid rises around it in broad bubble-like structures.

For the smaller viscosity, $\nu=0.001$ (bottom row), inertial effects are stronger and the shear layers along the sides of the descending spike become increasingly pronounced. By approximately $t=0.72$, the interface develops secondary roll-up structures consistent with Kelvin--Helmholtz-like shear instability. These structures continue to deform and entrain the surrounding fluid as the dense spike descends.

The proposed discretization remains numerically stable over the simulated time interval for both viscosity values and captures the qualitative transition from a strongly damped Rayleigh--Taylor evolution to a substantially more intricate low-viscosity interfacial flow.

\begin{figure}[h]
    \centering

    \begin{minipage}{0.10\textwidth}
        \centering
        \includegraphics[width=\linewidth]{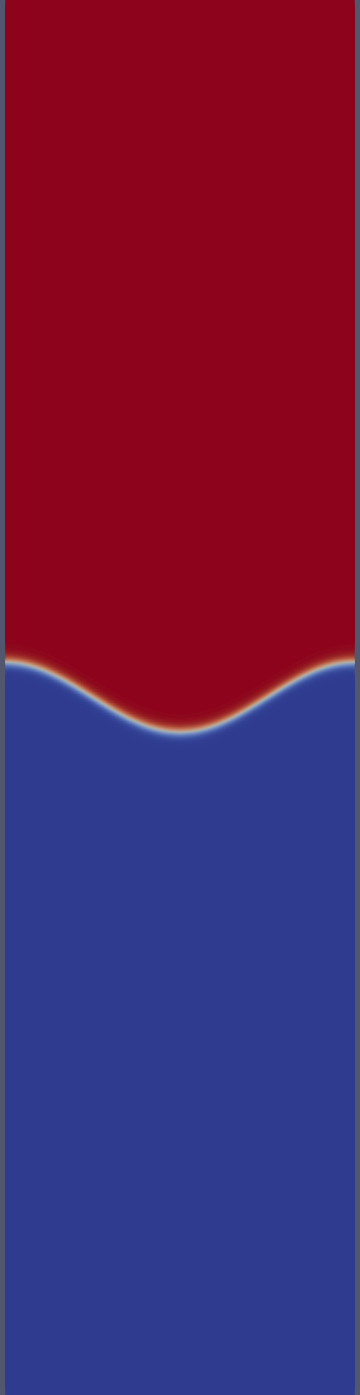}
        \centerline{$t=0.0$}
    \end{minipage}\hfill
    \begin{minipage}{0.10\textwidth}
        \centering
        \includegraphics[width=\linewidth]{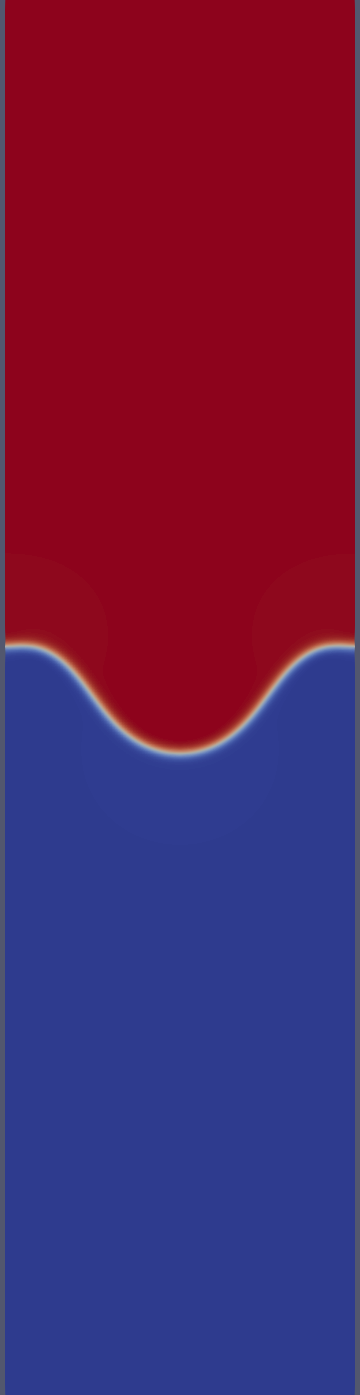}
        \centerline{$t=0.24$}
    \end{minipage}\hfill
    \begin{minipage}{0.10\textwidth}
        \centering
        \includegraphics[width=\linewidth]{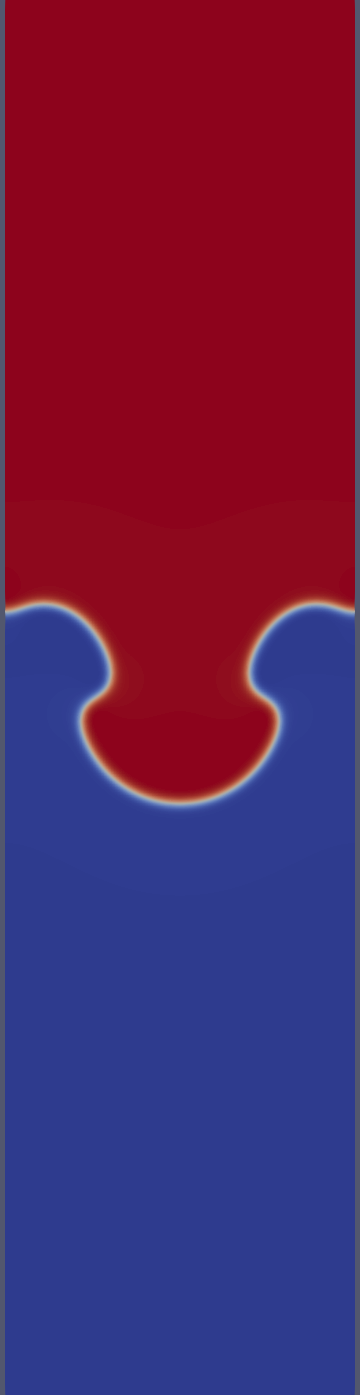}
        \centerline{$t=0.48$}
    \end{minipage}\hfill
    \begin{minipage}{0.10\textwidth}
        \centering
        \includegraphics[width=\linewidth]{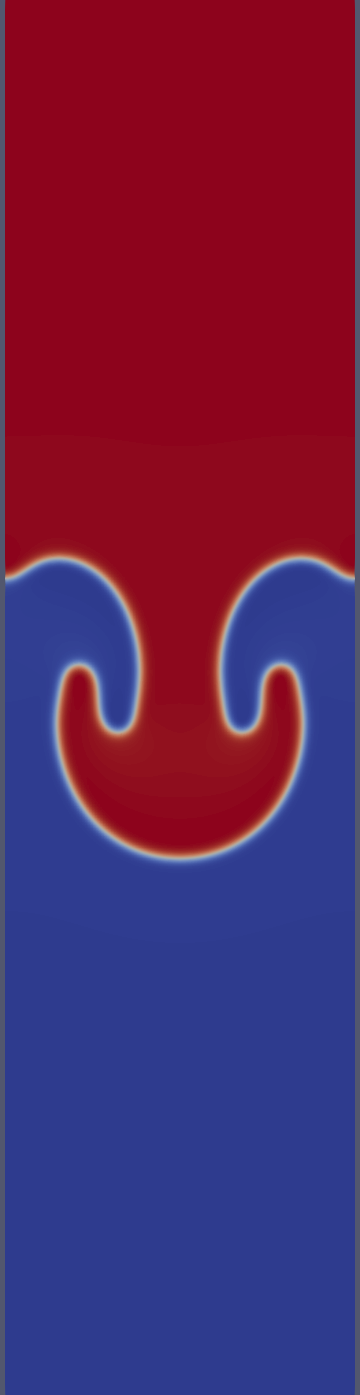}
        \centerline{$t=0.72$}
    \end{minipage}\hfill
    \begin{minipage}{0.10\textwidth}
        \centering
        \includegraphics[width=\linewidth]{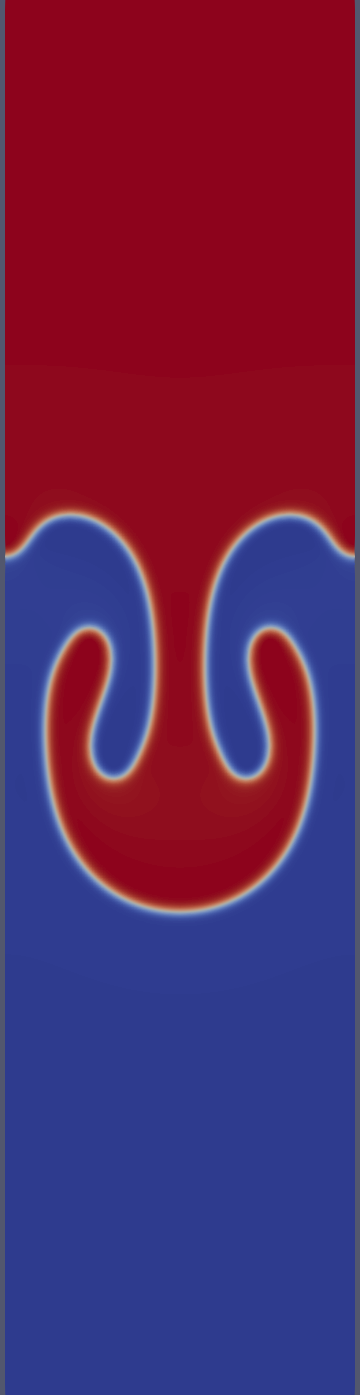}
        \centerline{$t=0.96$}
    \end{minipage}\hfill
    \begin{minipage}{0.10\textwidth}
        \centering
        \includegraphics[width=\linewidth]{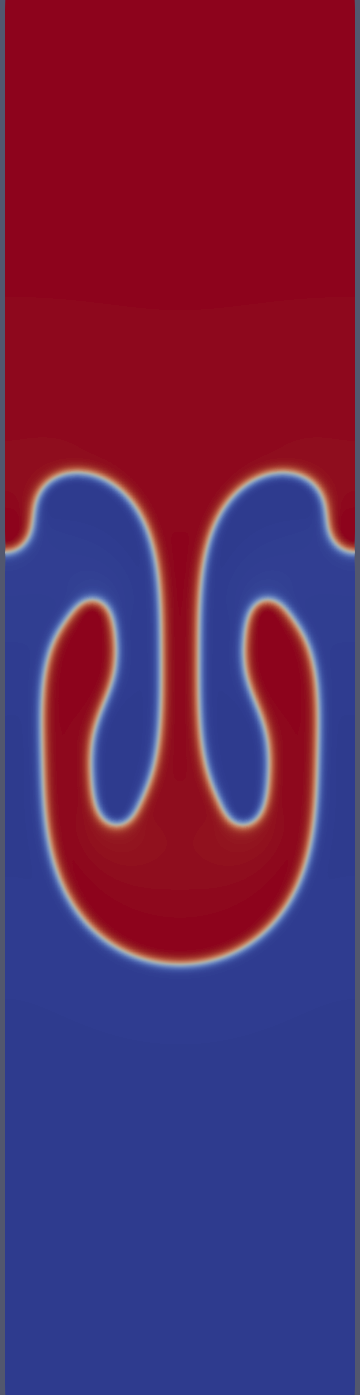}
        \centerline{$t=1.20$}
    \end{minipage}

    \vspace{0.1cm}

    \begin{minipage}{0.10\textwidth}
        \centering
        \includegraphics[width=\linewidth]{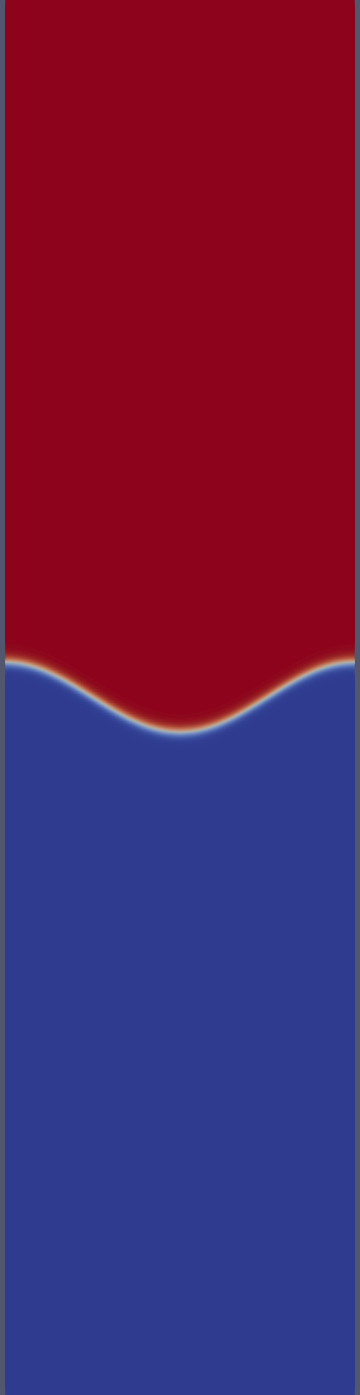}
        \centerline{$t=0.0$}
    \end{minipage}\hfill
    \begin{minipage}{0.10\textwidth}
        \centering
        \includegraphics[width=\linewidth]{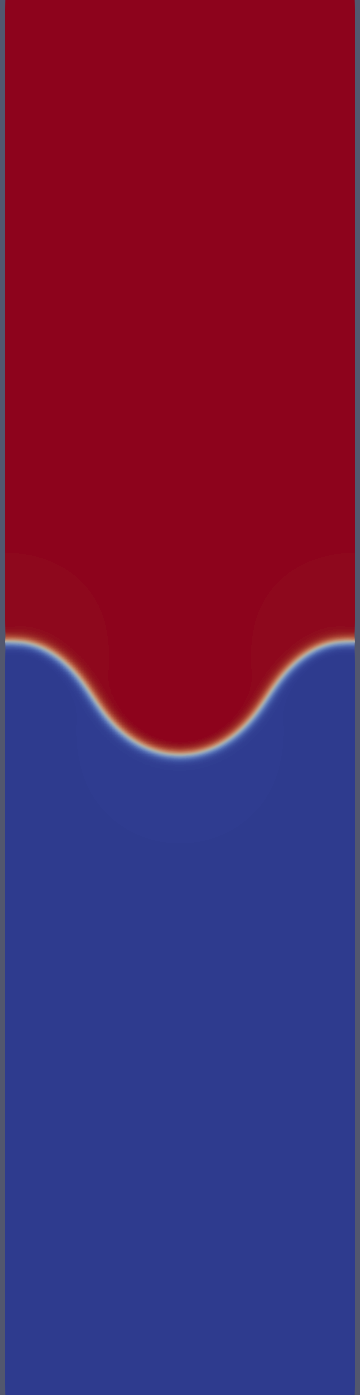}
        \centerline{$t=0.24$}
    \end{minipage}\hfill
    \begin{minipage}{0.10\textwidth}
        \centering
        \includegraphics[width=\linewidth]{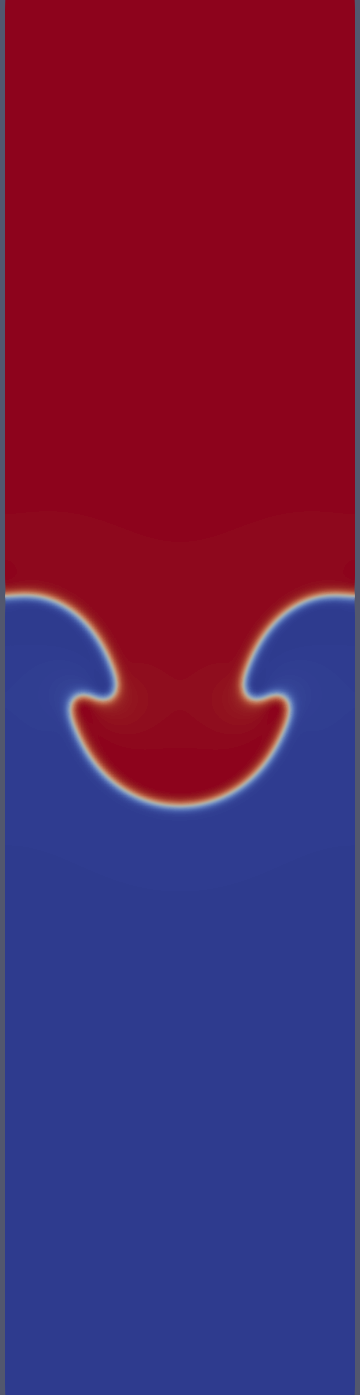}
        \centerline{$t=0.48$}
    \end{minipage}\hfill
    \begin{minipage}{0.10\textwidth}
        \centering
        \includegraphics[width=\linewidth]{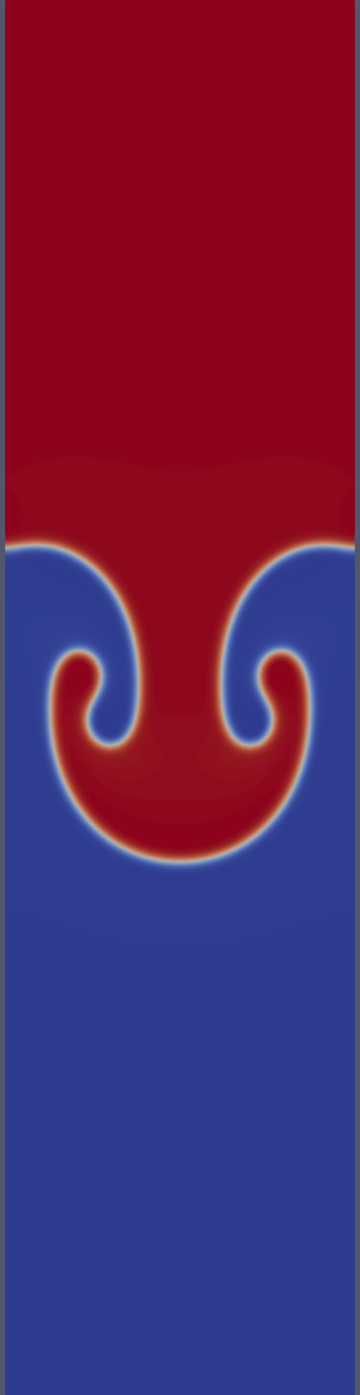}
        \centerline{$t=0.72$}
    \end{minipage}\hfill
    \begin{minipage}{0.10\textwidth}
        \centering
        \includegraphics[width=\linewidth]{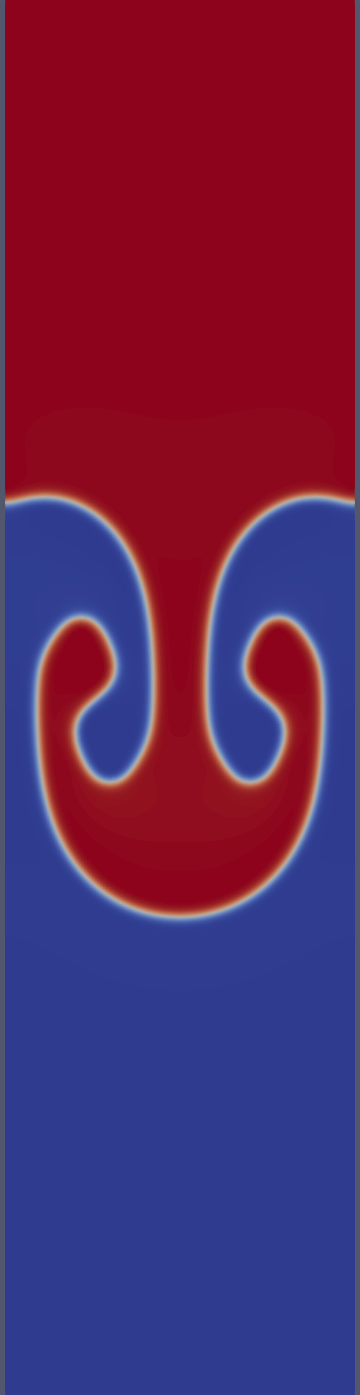}
        \centerline{$t=0.96$}
    \end{minipage}\hfill
    \begin{minipage}{0.10\textwidth}
        \centering
        \includegraphics[width=\linewidth]{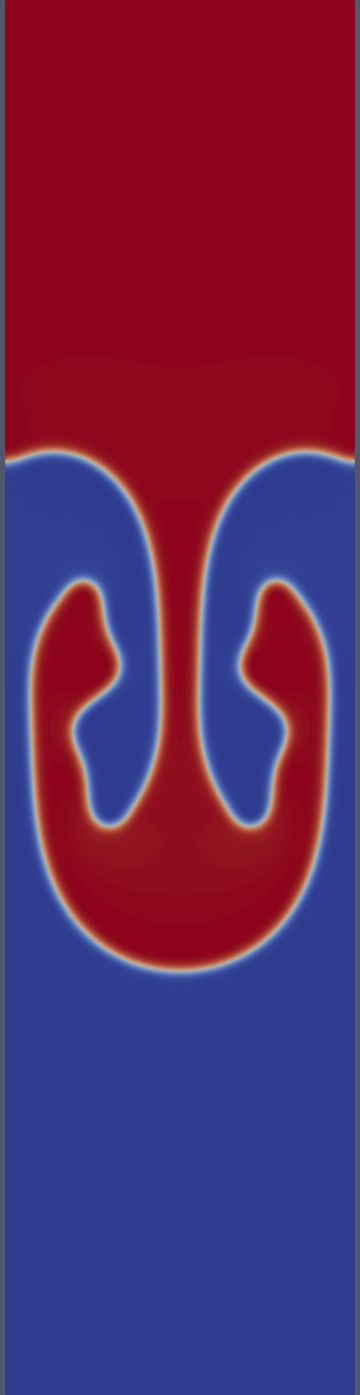}
        \centerline{$t=1.20$}
    \end{minipage}

    \caption{Evolution of the Rayleigh--Taylor instability for two viscous regimes. \textbf{Top row:} For $\nu=0.01$, stronger viscous damping produces a relatively smooth descending spike and broad rising structures. \textbf{Bottom row:} For $\nu=0.001$, stronger inertial and shear effects lead to secondary interfacial roll-up and more complex mushroom-like structures.}
    \label{fig:rt_instability}
\end{figure}

\section{Conclusion}
In this work, we developed a family of second-order, linear time-stepping methods for the matched-density Cahn--Hilliard--Navier--Stokes equations. By combining extrapolation, an auxiliary-variable reformulation of the nonlinear free-energy term, and curvature-stabilized interpolation, the proposed method avoids nonlinear iterations while retaining second-order temporal accuracy. A discrete energy estimate established the unconditional long-time stability of the scheme for $\theta\in(1/2,1]$ and $\epsilon\geq 0$.
The numerical experiments confirmed the expected temporal convergence rates and showed accurate mass conservation and energy dissipation. The spinodal decomposition, droplet relaxation, lid-driven cavity, and Rayleigh--Taylor simulations further showed that the method can capture complex interfacial dynamics over a range of flow regimes. Overall, the proposed scheme provides an efficient and stable alternative to fully nonlinear CHNS discretizations while preserving the principal physical and mathematical structures of the model.

\bibliographystyle{unsrturl} 
\bibliography{CHNS-ref} 

\end{document}